\documentclass{article}
\usepackage{amsfonts,amsmath,amsthm,amscd,amssymb,latexsym}
\usepackage{amsfonts,amscd,dsfont,amstext}
\usepackage[normalem]{ulem}
\input{xypic.tex} \xyoption{all}
\usepackage[pdftex]{graphicx}

 \newcommand{\C}{\mathbb{C}}
\newcommand{\Z}{\mathbb{Z}}

\newcommand{\pn}{{\mathbb{P}^n}}
\newcommand{\opn}{{\cal O}_{\mathbb{P}^n}}

\newcommand{\oh}{{\cal O}} 
\newcommand{\im}{{\rm im}~} 

\newcommand{\E}{\mathcal{E}}
\newcommand{\F}{\mathcal{F}}
\newcommand{\vv}{\mathbf{v}} \newcommand{\ww}{\mathbf{w}}
\DeclareMathOperator{\rk}{rk}

\newtheorem{theorem}{Theorem}[section]

\newtheorem{proposition}[theorem]{Proposition}
\newtheorem{lemma}[theorem]{Lemma}

\newtheorem{remark}[theorem]{Remark}

\newtheorem{definition}[theorem]{{\bf Definition}}

\title{Vector bundles on projective varieties and representations of quivers}
\author{Marcos Jardim and Daniela M. Prata \\ IMECC -- UNICAMP \\ 
Rua S\'ergio Buarque de Holanda, 651 \\ 
Campinas, SP, Brazil CEP 13083-859}

\begin{document}

\maketitle

\begin{abstract}
We present equivalences between certain categories of vector bundles on projective varieties, namely \emph{cokernel bundles}, \emph{Steiner bundles}, \emph{syzygy bundles}, and \emph{monads}, and full subcategories of representations of certain quivers. As an application, we provide decomposability criteria for such bundles.
\end{abstract}

\section{Introduction}

Vector bundles over algebraic varieties play a central role in algebraic geometry, and many interesting problems are still open. In particular, constructing indecomposable vector bundles on a variety $X$ with rank smaller than the $\dim X$ is not an easy task for certain choices of $X$, especially for projective spaces.

Monads are one of the most important tools for constructing such bundles; indeed, the majority of examples of low rank bundles on projective spaces, namely the Horrocks--Mumford bundle of rank $2$ on $\mathbb{P}^4$, Horrocks' parent bundle of rank $3$ on
$\mathbb{P}^5$, and the rank $2k$ instanton bundles on $\mathbb{P}^{2k+1}$, are obtained as cohomologies of certain monads.

The goal of this paper is to show that the theory of representations of quivers might also be an interesting tool for the construction of useful monads and cokernel bundles on projective varieties. More precisely, we present equivalences between certain categories of
vector bundles on projective varieties and full subcategories of
representations of certain quivers. In this way, we translate the problems
of constructing indecomposable vector bundles on $\pn$ with low rank into a
(possibly still very hard) problem of linear algebra. As an application of these 
results, we give decomposability criteria for cokernel bundles, syzygy bundles and monads.

Let us now present more precisely the results proved here, starting with \emph{cokernel bundles},
a class a vector bundles introduced by Brambilla in \cite{Bb}. Let $X$ be a nonsingular projective variety of dimension $n$, and let $\mathcal{E}$ and $\mathcal{F}$ be simple vector bundles on $X$ such that
\begin{itemize}
\item[(i)] ${\rm Hom}(\mathcal{F},\mathcal{E}) = {\rm Ext}^{1}(\mathcal{F},\mathcal{E}) =0$;
\item[(ii)] $\mathcal{E}^{\lor} \otimes \F$ is globally generated;
\item[(iii)] $\dim{\rm Hom}(\mathcal{E}, \mathcal{F})\le3$.
\end{itemize}
A \emph{cokernel bundle of type $(\mathcal{E},\mathcal{F})$} on $X$ is a vector bundle $\mathcal{C}$ with a resolution of the form
$$\xymatrix{
0 \ar[r] & \mathcal{E}^a \ar[r]^{\alpha} & \mathcal{F}^b \ar[r] & \mathcal{C} \ar[r] & 0
}.$$
We prove (cf. Thm \ref{Teo1} below):

\begin{theorem}
The category of cokernel bundles of type $(\mathcal{E},\mathcal{F})$ is equivalent to a full subcategory of the category of representation of the Kronecker quiver with 
$w = \dim{\rm{Hom}}(\mathcal{E,F})$ arrows:
$$ \xymatrix{
\bullet \ar@<1.8ex>[r]^1 \ar@<-1.8ex>[r]^{\vdots}_w & \bullet
} $$
\end{theorem}

As application of this equivalence, we obtain new proofs of simplicity and exceptionality criteria for cokernel bundles that were originally established by Brambilla in \cite[Thm. 4.3]{Bb} (cf. Thm \ref{decon} below) and Soares in \cite[Theorem 2.2.7]{Soa} (cf. Cor \ref{corSteiner} below).

Next, we consider 1$^{\rm st}$-syzygy bundles on projective spaces; recall that  \emph{syzygy bundles} are those given as kernel of surjective morphisms of the form
$$ \opn(-d_1)^{a_1}\oplus \cdots \oplus \opn(- d_m)^{a_m} \stackrel{\alpha}{\rightarrow} \opn^c .$$ 
Let $\mathcal{G}:=\ker\alpha$; we refer to \cite{Br} as a general reference on syzygy bundles.

The case $m=1$ can be regarded as a cokernel bundle; for the remainder of the paper, we focus on the case $m=2$, though it is not hard to generalize our results for $m>2$ (see Remark \ref{m>2} below). More precisely, we prove the following result, including a new decomposability criterion for syzygy bundles.

\begin{theorem} \label{thmsyzygy}
For any fixed integers $d_1>d_2>0$, there is a faithful functor from the category of representations of the quiver 
$$ \xymatrix{
\bullet \ar@<1.8ex>[r]^{1} \ar@<-1.8ex>[r]^{\vdots}_{w_1} & \bullet &
\ar@<1.8ex>[l]_{\vdots}^{w_2} \ar@<-1.8ex>[l]_{1} \bullet
}$$
to the category of syzygy bundles given by sequences of the form
\begin{equation} \label{syzmodel}
0\to \mathcal{G} \to \opn(-d_1)^{a_1}\oplus \opn(- d_2)^{a_2} \stackrel{\alpha}{\rightarrow}
\opn^c \to 0
\end{equation}
where $w_j=h^0(\opn(d_j))$, $j=1,2$. Moreover, if $a_1^2+ a_2^2 + c^2 - w_1 a_1 c - w_2 a_2 c > 1$, then $\mathcal{G}$ is decomposable.
\end{theorem}

Finally, we consider the relation between monads and representations of quivers. Recall that a \emph{monad} on a nonsingular projective variety $X$ is a complex of locally free sheaves of the form
\begin{eqnarray}\label{monadintro}
M^{\bullet}: \xymatrix{\mathcal{A}^{a} \ar[r] & \mathcal{B}^b \ar[r] & \mathcal{C}^c}
\end{eqnarray}
whose only nontrivial cohomology is the middle one, which we assume, in this paper, to also be a locally free sheaf. We prove:

\begin{theorem}\label{thm-monads}
If $\mathcal{A}$, $\mathcal{B}$ and $\mathcal{C}$ are simple vector bundles, then the category of monads of the form (\ref{monadintro}) is equivalent to a full subcategory of the category of representations of the quiver
$$ \xymatrix{
\bullet \ar@<1.8ex>[r]^1 \ar@<-1.8ex>[r]^{\vdots}_m &
\bullet \ar@<-1.8ex>[r]^{\vdots}_n \ar@<1.8ex>[r]^1 & \bullet
} $$
where $m = \dim {\rm Hom}(\mathcal{A},\mathcal{B})$ and $n = \dim {\rm Hom}(\mathcal{B},\mathcal{C})$. In addition, if $a^2 + b^2 + c^2 - m ab - n bc > 1$ then the cohomology sheaf of (\ref{monadintro}) is decomposable.
\end{theorem}

This generalizes the results of \cite{JS} (in particular, \cite[Thm 1.1]{JS}) concerning linear monads on $\pn$, i.e. when $\mathcal{A}=\opn(-1)$, $\mathcal{B}=\opn$ and $\mathcal{C}=\opn(1)$.

Furthermore, if $\mathcal{A}$, $\mathcal{B}$ and $\mathcal{C}$ are elements of distinct blocks of an $n$-block collection generating the bounded derived category $D^b(X)$ of coherent sheaves of $\mathcal{O}_X$-modules, then we also prove that the cohomology sheaf $\mathcal{E}$ of (\ref{monadintro}) is decomposable, if and only if the corresponding quiver representation is decomposable, cf. Theorem \ref{deccohomology}.

\bigskip

\noindent{\bf Notation.}
Throughout this paper, $\kappa$ denotes an algebraically closed field with characteristic zero, and $X$ is always a nonsingular projective variety over $\kappa$ of dimension $n$.   

\bigskip

\noindent{\bf Acknowledgements.} The first named author is partially supported by the CNPq grant number 400356/2015-5 and the FAPESP grant number 2014/14743-8. The second named author was supported by the FAPESP doctoral grant number 2007/07469-3 and the FAPESP post-doctoral grant number 2011/21398-7. Some of the results presented here were obtained in her PhD thesis. We thank Helena Soares for her help with the results in Section \ref{sec5}, and Rosa Maria Mir\'o-Roig for describing to us the monads considered in Section \ref{anexample}.

%%%%%%%%%%%%%%%%%%%%%%%%%%%%%%%%%%%%%%%%%%% QUIVERS %%%%%%%%%%%%%%%%%%%%%%%%%%%%%%%%%%%%%%%%%%%%%%%

\section{Preliminary definitions and results}

In this section we revise some key definitions and results on the theory of representations of quivers and on the derived category of coherent sheaves that will be relevant in the following sections.

\subsection{Representations of quivers}\label{quivers}

We begin by revising some basic facts about representations of quivers. Recall that a {\it quiver} $Q$ consists on a pair $(Q_0, Q_1)$ of sets where $Q_0$ is the set of vertices and $Q_1$ is the set of arrows and a pair of maps $t, h: Q_1 \rightarrow Q_0$ the tail and head maps. An example is the \emph{Kronecker quiver}, denoted $K_{w}$, which consists of $2$ vertices and $w$ arrows.

\begin{equation} \label{kronquiver}
\xymatrix{
\bullet \ar@<1.8ex>[r]^1 \ar@<-1.8ex>[r]^{\vdots}_w & \bullet
}
\end{equation}

A {\it representation} $R = (\{V_i\}, \{A_a\})$ of $Q$ consists of a collection of finite dimensional $\kappa$-vector spaces $\{V_i; i \in Q_0 \}$ together with a collection of linear maps $ \{A_a : V_{t(a)} \rightarrow V_{h(a)}; a \in Q_1\}$.  A {\it morphism} $f$ between two representations $R_{1} = (\{V_i\}, \{A_a\})$ and $ R_{2} = (\{W_i\}, \{B_a\})$ is a collection of linear maps $\{f_i\}$ such that for each $a \in Q_1$ the diagram bellow is commutative
$$\xymatrix{V_{t(a)} \ar[r]^{A_a} \ar[d]_{f_{t(a)}} & V_{h(a)} \ar[d]^{f_{h(a)}}\\
W_{t(a)} \ar[r]_{B_a} & W_{h(a)}}$$ 
With these definitions, representations of $Q$ form an abelian category hereby denoted by $\mathfrak{R}(Q)$. 

Given a representation $R \in \mathfrak{R}(Q)$, we associate a vector $\vv \in \Z^{Q_0}$ called {\it{dimension vector}}, whose entries are $\vv_i = \dim V_i$. 

The \emph{Euler form} on $\mathbb{Z}^{Q_0}$ is a bilinear form associated to $Q$, given by
$$< \vv, \ww>  = \sum_{i \in Q_0} \vv_i \ww_i  - \sum_{a \in Q_1} \vv_{t(a)} \ww_{h(a)}.$$ 
The \emph{Tits form} is the corresponding quadratic form, given by
$$q(\vv) = <\vv, \vv>.$$
For instance, the Tits form of the Kronecker quiver with $w$ arrows is given by
\begin{eqnarray}\label{titsKronecker}
{q_{w}(\vv) = a^2 + b^2 - wab}, ~~ \vv = (a,b) \in \Z^2.
\end{eqnarray} 

\begin{definition}
A vector $\vv \in \Z^{Q_0}$ is a {\rm{root}} if there is an indecomposable representation $R$ of $Q$ with dimension vector $\vv$. Moreover, $\vv$ is a {\rm Schur root}
if there is a representation $R$ of $Q$ with dimension vector $\vv$ satisfying ${\rm Hom}(R,R)=\kappa$.
\end{definition}

Clearly, every Schur root is a root; note also that the condition ${\rm Hom}(R,R)=\kappa$ is an open condition in the affine space 
$$ \oplus_{a \in Q_1} {\rm Hom} (\kappa^{\vv_{t(a)}}, \kappa^{\vv_{h(a)}}) $$
of all representations with fixed dimension vector $\vv$. Thus if $\vv$ is a Schur root, then
${\rm Hom}(R,R)=\kappa$ for a generic representation with dimension vector $\vv$. In particular, if $\vv$ is a Schur root, then generic representation with dimension vector $\vv$ is indecomposable. A reference for generic representations and Schur roots is \cite{S}. For more information about roots and root systems, we refer to $\cite{K}$.

The following two facts will be very relevant in what follows.  The first one follows from Kac's theory of infinite root systems \cite{K}. 

\begin{proposition}\label{Kac}
Let  $Q$ be a quiver with Tits form $q$. If $\vv$ is a dimension vector satisfying $q(\vv) > 1$, then every representation with dimension vector $\vv$ is decomposable.
\end{proposition}

The second fact follows from \cite[Prop 1.6]{K} and \cite[Thm 4.1]{Schof}.

\begin{proposition}\label{schur} 
Let $Q$ be the Kronecker quiver with $w \geq 3$, and let $\vv\in\mathbb{Z}^{2}$ be a dimension vector. If $q_{w}(\vv) \leq 1 $, then $\vv$ is a Schur root.
\end{proposition}

%%%%%%%%%%%%%%%%%%%%%%%%%%%%%%%%%%%%%%%

\subsection{Derived categories}

In \cite{SM}, Mir\'o-Roig and Soares gave a cohomological characterisation of Steiner bundles and later Marques and Soares \cite{SMM}, gave a cohomological characterisation of a class of bundles given as cohomology of monads. Both results will be relevant for us, so we review them here.

Let $D^b(X)$ be the bounded derived category of the abelian category of coherent sheaves of $\mathcal{O}_X$-modules. An \emph{exceptional collection} is an ordered collection
$(\mathcal{F}_0, \cdots, \mathcal{F}_m)$ of objects of $D^b(X)$ such that
$$ {\rm Hom}^0_{D^b(X)}(\mathcal{F}_i, \mathcal{F}_i) \simeq \kappa, \;\;
{\rm Ext}^p(\mathcal{F}_i,\mathcal{F}_i) = 0, \;\; \mbox{for all}\;\;  p\geq1, $$
$$ {\rm Ext}^{p}(\mathcal{F}_i,\mathcal{F}_j) = 0\; \; \mbox{for all}\; i>j,\; \mbox{and}\; p\geq0. $$
In addition, if
$${\rm Ext}^p(\mathcal{F}_i,\mathcal{F}_j)=0 \;\; {\rm for}\;\; i \leq j \;\; {\rm and}\;\; p\neq0\,,$$ 
then $(\mathcal{F}_0, \dots, \mathcal{F}_m)$ is called a \emph{strongly exceptional collection}. It is a \emph{full (strongly) exceptional collection} if it generates $D^b(X)$.

An exceptional collection $(\mathcal{F}_0, \cdots, \mathcal{F}_m)$ is called  a \emph{block} if 
$$ {\rm Ext}^p(\mathcal{F}_j ,\mathcal{F}_i) = 0
\;\; \forall \;\; p\geq 0\;\; \mbox{and}\;\; i \neq j. $$
An \emph{$m$-block collection} of type $(t_0,\dots,t_m)$ is an exceptional collection $\boldsymbol{B}=(\boldsymbol{\mathcal{F}}_0,\dots, \boldsymbol{\mathcal{F}}_m)$ where each $\boldsymbol{\mathcal{F}}_i = (\mathcal{F}^i_1, \dots, \mathcal{F}^i_{t_i})$ is a block.

\begin{definition}
Let $\boldsymbol{B}=(\boldsymbol{\mathcal{F}}_0, \dots, \boldsymbol{\mathcal{F}}_m)$
be an $m$-block collection of type \linebreak $(t_0, \dots, t_m)$. The {\rm left dual $m$-block collection} of {\bf B} is the $m$-block collection ${}^{\vee}\bf{B}$ of type $(u_0, \dots, u_m) $ with $u_i = u_{m-i}$
$$ {}^{\vee}{\bf B} = (\boldsymbol{\mathcal H}_0, \dots, \boldsymbol{\mathcal H}_m)  =
({\mathcal H}^0_1, \dots, {\mathcal H}^0_{u_0}, \dots, {\mathcal H}^m_1, \dots,{\mathcal H}^m_{u_m}) $$ 
where
$$ {\rm Hom}^k_{D^b(X)}({\mathcal H}^i_j,{\mathcal F}^l_p) = 0 $$
for all indices, with the only exception
$$ {\rm Ext}^i(\mathcal{H}^i_j,\F^{m-i}_j) \simeq \kappa .$$
These conditions uniquely determine ${}^{\vee}{\bf B}$.
\end{definition}

We are now able to define Steiner bundles in the sense of \cite{SM} and state their cohomological characterisation.

\begin{definition}\label{defsteiner}
A vector bundle $\mathcal{S}$ on $X$ is a {\rm Steiner bundle} of type $(\F_0, \F_1)$ if it is given by a short exact sequence of the form
$$\xymatrix{0 \ar[r] & \F_0^a \ar[r]^{\alpha} &  \F_1^b \ar[r] & \mathcal{S} \ar[r] & 0 }$$
such that $a,b \geq 1$ and $(\F_0,\F_1)$ is an ordered pair of vector bundles on $X$ satisfying
\begin{itemize}
\item[$(i)$] $(\F_0,\F_1)$ is strongly exceptional;
\item[$(ii)$] $\F_0^\lor \otimes \F_1$ is globally generated.
\end{itemize} 
\end{definition}

The cohomological characterisation is the following, cf. \cite[Thm 2.4]{SM}.

\begin{theorem}\label{carsteiner}
Let $X$ be a smooth projective variety of dimension $n$ with an $n$-block collection 
${\bf B}= (\boldsymbol{\mathcal{F}}_0, \dots, \boldsymbol{\mathcal{F}}_n)$,
$\boldsymbol{\mathcal{F}}_i = (\F^i_1, \dots, \F^i_{t_i})$ of locally free sheaves which generate $D^b(X)$, and let ${}^{\vee}{\bf B}$ be its left dual basis.  Let $\F^i_{i_0} \in \boldsymbol{\mathcal{F}}_i$ and
$\F^j_{j_0} \in \boldsymbol{\mathcal{F}}_j$, where $0 \leq i < j \leq n$ and $ 1 \leq i_0 \leq a_i$, 
$1 \leq j_0 \leq a_j$, and let $\mathcal{S}$ be a locally free sheaf on $X$. Then $\mathcal{S}$ is a Steiner bundle of type $(\F^i_{i_0},\F^j_{j_0})$ given by the short exact sequence
$$ \xymatrix{0 \ar[r] & (\F^i_{i_0})^a \ar[r] & (\F^{j}_{j_0})^b \ar[r] & \mathcal{S} \ar[r] & 0}$$
if and only if $(\F^i_{i_0})^\lor \otimes \F^j_{j_0}$ is globally generated and all
${\rm Ext}^l(\mathcal{H}^m_p,\mathcal{S})$ vanish, with the only exceptions of 
\begin{equation}
\dim {\rm Ext}^{n-i-1}(\mathcal{H}^{n-i}_{i_0}, \mathcal{S}) = a ~~ {\rm and}
~~\dim {\rm Ext}^{n-j}(\mathcal{H}^{n-j}_{j_0}, \mathcal{S}) = b.
\end{equation} 
\end{theorem}
 
Now we turn our attention to the cohomological characterisation for the bundles obtained as cohomology of monads, due to Marques and Soares in \cite{SMM}. 

\begin{definition}\label{defmon}
A {\rm monad} $M^{\bullet}$ on a smooth projective variety $X$ is a complex of locally free coherent sheaves on $X$
$$ M^{\bullet}: \xymatrix{\mathcal{A} \ar[r]^{\alpha} & \mathcal{B} \ar[r]^{\beta} & \mathcal{C} } $$
such that $\alpha$ is injective, $\beta$ is surjective; the coherent sheaf
$\mathcal{E} = \ker{\beta} / {\rm im}\,\alpha$ is called the {\rm cohomology} of $M^{\bullet}$.
\end{definition} 

The following two definitions are important for the main result we would like to present.

\begin{definition}
Let ${\bf B}= (\boldsymbol{\mathcal{F}}_0, \cdots, \boldsymbol{\mathcal{F}}_m)$,
$\boldsymbol{\mathcal{F}}_i = (\F^i_1, \cdots, \F^i_{t_i})$, be an $m$-block collection. 
A coherent sheaf $\mathcal{E}$ on $X$ has \emph{natural cohomology with respect to} ${\bf B}$ if for each $0 \leq p \leq m$ and $1 \leq j \leq t_p$ there is at most one $q \geq 0$ such that
${\rm Ext}^q(\F^p_j,\mathcal{E}) \neq 0$.
\end{definition}

\begin{definition}\label{Beilinson}
Let $X$ be a smooth projective variety with an $m$-block collection ${\bf B}=(\boldsymbol{\mathcal{F}}_0, \cdots, \boldsymbol{\mathcal{F}}_m)$, $\boldsymbol{\mathcal{F}}_i = (\mathcal{F}^i_1, \cdots, \mathcal{F}^i_{t_i})$ of coherent sheaves on $X$. A {\rm Beilinson monad} for $\mathcal{E}$ is a bounded complex $G^{\bullet}$ in $D^b(X)$ whose terms are finite direct sums of elements of ${\bf B}$ and whose cohomology is $\mathcal{E}$, that is,
$$\bigoplus_{i \in \Z} H^i(G^{\bullet}) = H^0(G^{\bullet}) = \mathcal{E}.$$
\end{definition}

The next result tell us when a coherent sheaf $\mathcal{E}$ on $X$ is isomorphic to a Beilinson monad $G^{\bullet}$, see \cite[Cor 1.7]{SMM}.

\begin{lemma}\label{beilinsonlemma}
Let $X$ be a smooth projective variety of dimension $n$ with an $n$-block collection ${\bf B} = (\boldsymbol{\mathcal{F}}_0, \cdots, \boldsymbol{\mathcal{F}}_n)$ generating $D^b(X)$. Let ${}^\vee {\bf B} = (\boldsymbol{\mathcal{H}}_0, \cdots, \boldsymbol{\mathcal{H}}_n)$ with $\boldsymbol{\mathcal{H}}_i = (\mathcal{H}^i_1, \cdots, \mathcal{H}^i_{u_i})$, be its left dual $n$-block collection. Then each coherent sheaf $\mathcal{E}$ on $X$ is isomorphic to a Beilinson monad $G^{\bullet}$ with each $G^r$ given by 
$$ G^r = \bigoplus_{p,q}{\rm Ext}^{n-q+r}(\mathcal{H}_p^{n-q}, \mathcal{E}) \otimes \F^q_p. $$
\end{lemma}

The cohomological characterisation for monads is the following, cf. \cite[Thm 2.2]{SMM}.

\begin{theorem}\label{carmon}
Let $X$ be a nonsingular projective variety of dimension $n$, and let
${\bf B} = (\boldsymbol{\mathcal{F}}_0, \cdots, \boldsymbol{\mathcal{F}}_n)$, where 
$\boldsymbol{\mathcal{F}}_i = (\mathcal{F}^i_1, \cdots, \mathcal{F}^i_{t_i})$, be an $n$-block collection of coherent sheaves on $X$ generating $D^b(X)$. Let ${}^{\vee}{\bf B} $ be its left dual $n$-block collection, and let $\mathcal{F}^i_{i_0}$, $\mathcal{F}^j_{j_0}$, and $\mathcal{F}^k_{k_0}$ be elements of the blocks $\boldsymbol{\mathcal{F}}_i, \boldsymbol{\mathcal{F}}_j$ and 
$\boldsymbol{\mathcal{F}}_k$, respectively, with $0 \leq i < j< k \leq n$.

A torsion-free sheaf $\mathcal{E}$ on $X$ is the cohomology sheaf of a monad of the form
\begin{equation} \label{moncarmon}
M^{\bullet}:
\xymatrix{({\mathcal{F}}^i_{i_0})^{a} \ar[r] & ({\mathcal{F}}^j_{j_0})^b \ar[r] & ({\mathcal{F}}^{k}_{k_0})^c }
\end{equation}
for some $b\ge 1$ and $a,c\ge0$ if and only if  $\mathcal{E}$ has:
\begin{itemize}
\item[(1)] rank $b\cdot\rk({\mathcal{F}}^j_{j_0}) - a\cdot\rk({\mathcal{F}}^i_{i_0}) - c\cdot\rk({\mathcal{F}}^{k}_{k_0})$; 
\item[(2)] Chern polynomial $c_t(\mathcal{E}) = c_t(F^j_{j_0})^b c_{t}(F^i_{i_0})^{-a}c_t(F_{k_0}^k)^{-c} $;
\item[(3)] natural cohomology with respect to ${}^{\vee}{\bf B}$.
\end{itemize}
\end{theorem}

\begin{remark}\rm
The original statement of \cite[Thm 2.2]{SMM} requires $a,b,c \geq 1$. However, following the same steps of the proof of \cite[Thm 2.2]{SMM}, one can prove that the result also holds for $a,c\geq0$; in other words, one can allow for degenerate monads.
\end{remark}

This result will be very useful in the last section of this paper, in which we study the decomposability of sheaves given by the cohomology of monads of the above form.

%%%%%%%%%%%%%%%%%%%%%%%%%%%%%%%%%%%%%%%%%%%%%%%%%%%%%%%%%%%%%%%%%%%%%%%%%%%%%%%%%%%%%%%%%%%%%

\section{Cokernel and Steiner bundles}

In this section we explain the relation between cokernel and Steiner bundles and representations of the Kronecker quiver.

\subsection{Cokernel bundles}

Let $\mathcal{E}$ and $\mathcal{F}$ be vector bundles on a nonsingular projective variety $X$ of dimension $n \geq 2$, satisfying the following conditions:

\begin{itemize}
\item[$(1)$] $\mathcal{E}$ and $\mathcal{F}$ are simple, that is, $\rm{Hom}(\mathcal{E},\mathcal{E}) = \rm{Hom}(\mathcal{F},\mathcal{F}) = \kappa$;
\item[$(2)$] ${\rm Hom}(\mathcal{F},\mathcal{E}) =0;$
\item[$(3)$] ${\rm Ext}^{1}(\mathcal{F},\mathcal{E}) = 0$;
\item[$(4)$] the sheaf $\mathcal{E}^{\lor} \otimes \F$ is globally generated;
\item[$(5)$] $W = {\rm Hom}(\mathcal{E}, \mathcal{F})$ has dimension $w \geq 3$.
\end{itemize}

The next definition is due to Brambilla \cite{Bb}.

\begin{definition}
A \emph{cokernel bundle of type $(\mathcal{E},\mathcal{F})$} on $\pn$ is a vector bundle $\mathcal{C}$ with resolution of the form
\begin{eqnarray}\label{cok}
\xymatrix{ 0  \ar[r] & \mathcal{E}^{a} \ar[r]^{\alpha} & \mathcal{F}^{b} \ar[r] & \mathcal{C} \ar[r]& 0 }
\end{eqnarray}  
where $\mathcal{E}, \mathcal{F}$ satisfy the conditions (1) through (5) above, $a\ge 0$ and
$b\cdot\rk(\mathcal{F}) - a\cdot\rk(\mathcal{E}) \geq n$.
\end{definition}

Cokernel bundles of type $(\mathcal{E},\mathcal{F})$ form a full subcategory of the category of coherent sheaves on $X$; this category will be denoted by $\mathfrak{C}_X(\mathcal{E},\mathcal{F})$. 

%Below we have some examples of cokernel bundles.
%
%\begin{enumerate}
%\item The bundle given by exact sequence
%$$\xymatrix{0 \ar[r] & \opn(-d)^a \ar[r]^{\alpha} & \opn^b \ar[r] & \mathcal{C} %\ar[r] & 0}$$
%with $d \geq 1$. The map $\alpha$ is given by a matrix $b \times a$ of forms %of degree $d$.
%
%\item The bundle
%$$\xymatrix{ 0 \ar[r] & \Omega^p(p)^c \ar[r]^{\alpha} & \opn^b \ar[r] & %\mathcal{C} \ar[r] & 0}$$
%where $0 < p \leq n$ and $\alpha$ is a matrix given by $p$-forms.
%\end{enumerate}

Let us now see how cokernel bundles are related to quivers. Fix a basis $\boldsymbol{\sigma} = \{\sigma_1, \cdots, \sigma_w\}$ of ${\rm Hom}(\mathcal{E},\mathcal{F})$. 

\begin{definition}
A representation $R = (\{\kappa^a, \kappa^b\}, \{A_i\}_{i=1}^w)$ of $K_w$ is $(\mathcal{E},\mathcal{F}, \boldsymbol{\sigma})$-\emph{globally injective} when the map 
$$ \alpha(P) := \sum_{i=1}^w A_i \otimes \sigma_i(P) ~:~ 
\kappa^a\otimes \mathcal{E}_P \to \kappa^b\otimes \mathcal{F}_P $$
is injective for every $P \in X$; here, $\mathcal{E}_P$ and $\mathcal{F}_P$ denote the fibers of $\mathcal{E}$ and $\mathcal{F}$ over the point $P$, respectively. 
\end{definition}

$(\mathcal{E},\mathcal{F}, \boldsymbol{\sigma})$-globally injective representations of $K_w$ form a full subcategory of the category of representations of $K_w$; we denote it by $\mathfrak{R}(K_w)^{gi}$. From now on, since $(\mathcal{E},\mathcal{F}, \boldsymbol{\sigma})$ are fixed, we will just refer to globally injective representations. It is a simple exercise to establish the following properties of $\mathfrak{R}(K_w)^{gi}$.

\begin{lemma}\label{subob}
The category $\mathfrak{R}(K_w)^{gi}$ is closed under sub-objects, i.e. every subrepresentation $R'$ of a representation $R$ in $\mathfrak{R}(K_w)^{gi}$ is also in $\mathfrak{R}(K_w)^{gi}$.
\end{lemma}

\begin{lemma}
The category $\mathfrak{R}(K_{w})^{gi}$ is closed under extensions and under direct summands, that is, respectively:
\begin{itemize}
\item[(i)] if $R_1, R_2 \in \mathfrak{R}(K_w)^{gi}$ and 
$$ \xymatrix{0 \ar[r] & R_1 \ar[r] & R \ar[r] & R_2 \ar[r] & 0} $$
is a short exact sequence in $\mathfrak{R}(K_w)^{gi}$, then $R \in \mathfrak{R}(K_w)^{gi}$;
\item[(ii)] if $R \in \mathfrak{R}(K_w)^{gi}$ with $R \simeq R_1 \oplus R_2$, then 
$R_i \in \mathfrak{R}(K_w)^{gi}$, $i=1,2$.
\end{itemize}
\end{lemma}

Our next result relates the category of globally injective representations of $K_w$ to the category of cokernel bundles.

\begin{theorem}\label{Teo1}
For every choice of basis $\boldsymbol{\sigma}$ of ${\rm Hom}(\E,\F)$, there is an equivalence between $\mathfrak{R}(K_w)^{gi}$, the category of $(\E,\F, \boldsymbol{\sigma})$-globally injective representations of $K_w$, and $\mathfrak{C}_X(\E,\F)$, the category of cokernel bundles of type $(\E,\F)$.
\end{theorem}

\begin{proof}
Given a basis $\boldsymbol{\sigma}$ of ${\rm Hom}(\E,\F)$, we construct a functor 
$$ {\bf L}_{\boldsymbol{\sigma}} : \mathfrak{R}(K_{w})^{gi} \rightarrow \mathfrak{C}_X(\E,\F) $$
and show that it is essentially surjective and fully faithful. 

Let $R = (\{\kappa^a, \kappa^b\}, \{A_i\}_{i=1}^w)$ be a globally injective representation of $K_{w}$. Define a map $\alpha : \E^{a} \rightarrow \F^{ b}$ given by 
$$ \alpha = A_1 \otimes  \sigma_1  + \cdots + A_w \otimes \sigma_w. $$
Since $R$ is globally injective, we have that $\dim {\rm coker}\, \alpha(P) = b\cdot\rk(\F) - a\cdot\rk(\E)$ for each $P \in X$. Therefore $\alpha$ is injective as a map of sheaves, and $\mathcal{C}:={\rm coker}\,\alpha$ is a cokernel bundle.

Now given two globally injective representations
$$R_{1} = (\{\kappa^a, \kappa^b\}, \{A_i\}_{i=1}^w) ~~{\rm and}~~  
R_2 = (\{\kappa^c, \kappa^d\}, \{B_i\}_{i=1}^w), $$
and a morphism $f=(f_1, f_2)$ between them, let ${\bf L}_{\boldsymbol{\sigma}}(R_1) = \mathcal{C}_1, {\bf L}_{\boldsymbol{\sigma}}(R_2) = \mathcal{C}_2$ be the cokernel bundles and $\alpha_1$, $\alpha_2$ the maps associated to $R_{1}$ and $R_{2}$, respectively. We want to define a morphism ${\bf L}_{\boldsymbol{\sigma}}(f): \mathcal{C}_1 \rightarrow \mathcal{C}_2$.

Since we have $f_1: \kappa^a \rightarrow \kappa^c, f_2: \kappa^b \rightarrow \kappa^d$, we have maps
$f_{1}^{'} = f_1 \otimes \mathds{1}_{\E} \in {\rm{Hom}}(\mathcal{E}^a,\E^c)$ and
$f_{2}^{'} = f_2 \otimes  \mathds{1}_{\F} \in {\rm Hom}(\mathcal{F}^b, \F^d)$. Consider the diagram
\begin{eqnarray}\label{diagg}
\xymatrix{ 0 \ar[r] & \E^a \ar[d]_{f_{1}^{'}} \ar[r]^{\alpha_1} & \F^b \ar[d]_{f_{2}^{'}} \ar[r]^{\pi_1}&  \mathcal{C}_1 \ar[r] \ar@{-->}[d] & 0\\
0 \ar[r] & \E^c \ar[r]^{\alpha_2} & \F^d \ar[r]^{\pi_2} & \mathcal{C}_2 \ar[r] & 0}
\end{eqnarray} 
where $\pi_1, \pi_2$ are the projections. Applying the left exact contravariant functor
${\rm Hom}( - , \mathcal{C}_2)$ to the upper sequence on (\ref{diagg}) we find a map
$\phi \in {\rm Hom}(\mathcal{C}_1,\mathcal{C}_2)$ and we define ${\bf L}_{\boldsymbol{\sigma}}(f) := \phi$.

Now given $\mathcal{C}$ an object of $\mathfrak{C}_X(\E,\F)$ we take $\alpha  = \sum_{i=1}^{w} A_i \otimes \sigma_i$, with \linebreak $A_i \in {\rm Hom}(\kappa^a, \kappa^b), i=1, \cdots, w.$ Hence
$R = (\{\kappa^a, \kappa^b\}, \{A_i\}_{i=1}^w)$ is a globally injective representation of $\mathfrak{R}(K_{w})$ such that ${\bf L}_{\boldsymbol{\sigma}}(R) = \mathcal{C}$. Therefore ${\bf L}_{\boldsymbol{\sigma}}$ is essentially surjective.

Finally, we need to prove that ${\bf L}_{\boldsymbol{\sigma}}$ is fully faithful.  To check that it is full,  given
$\phi \in {\rm Hom}({\bf L}_{\boldsymbol{\sigma}}(R_1), {\bf L}_{\boldsymbol{\sigma}}(R_2))$ we want $f = (f_1,f_2) \in {\rm Hom}(R_1,R_2)$ such that ${\bf L}_{\boldsymbol{\sigma}}(f)= \phi$. Let $\tilde{\phi} = \phi \pi_1 \in {\rm Hom}(\F^b, \mathcal{C}_2)$. Let us apply the left exact covariant functor
${\rm Hom}(\F^b, -)$ to the lower sequence in diagram (\ref{diagg1}) below:

\begin{eqnarray}\label{diagg1}
\xymatrix{ 0 \ar[r] & \E^a \ar@{-->}[d]_{f_{1}^{'}} \ar[r]^{\alpha_1} & \F^b \ar@{-->}[d]_{f_{2}^{'}} \ar[r]^{\pi_1}&  \mathcal{C}_1 \ar[r] \ar[d]^\phi & 0\\
0 \ar[r] & \E^c \ar[r]^{\alpha_2} & \F^d \ar[r]^{\pi_2} & \mathcal{C}_2 \ar[r] & 0}
\end{eqnarray}

we conclude that 
\begin{eqnarray}\label{eqq}
\rho_2:  {\rm Hom}(\F^b,\F^d ) \rightarrow {\rm Hom}(\F^b, \mathcal{C}_2)
\end{eqnarray}
is an isomorphism since ${\rm Hom}(\F^b, \E^c) = {\rm Ext}^{1}(\F^b, \E^c) = 0$.
It follows that there is a morphism $f_{2}^{'}\in {\rm Hom}(\F^b,\F^d )$ such that
$$ \rho_2 (f^{'}_2) = \pi_2 f^{'}_2  = \phi \pi_1 $$
with $f^{'}_2 = f_2 \otimes \mathds{1}_{\F}$ and $f_2 \in {\rm Hom}(\kappa^b, \kappa^d)$.

Consider $\tilde{\tilde{\phi}} = f_{2}^{'} \alpha_1 \in {\rm Hom}(\E^a, \F^d)$. Applying the left exact covariant functor
${\rm Hom}(\E^a, -)$ to the lower sequence on (\ref{diagg1}) we get
$$ \xymatrix{ 
0\ar[r] & {\rm Hom}(\E^a, \E^c) \ar[r]^{\gamma_1}& {\rm Hom}(\E^a,\F^d) \ar[r]^-{\gamma_2} &  {\rm Hom}(\E^a, \mathcal{C}_2) \ar[r] &\cdots
}$$

Once we have an exact sequence,
$$ \gamma_2(f^{'}_2\alpha_1) = \pi_2 f_2^{'}\alpha_1 = \phi \pi_1 \alpha_1 = 0 $$
then $f^{'}_2 \alpha_1 \in \ker{\gamma_2} = {\rm im}\, \gamma_1 $, and there is a map $f^{'}_1 \in {\rm Hom}(\E^a, \E^c)$ such that
$\gamma_1 (f_{1}^{'}) = \alpha_2 f^{'}_{1} = f^{'}_2 \alpha_1$ and $f^{'}_1 = f_1 \otimes \mathds{1}_{\E}$ with
$f_1 \in {\rm Hom}(\kappa^a, \kappa^c)$.

Since $\alpha_1 = \sum_{i=1}^{w} A_i \otimes \sigma_i$, $\alpha_2 = \sum_{i=1}^{w} B_i \otimes \sigma_i$, $\alpha_2 f^{'}_1 = f^{'}_2 \alpha_1$, and $\boldsymbol{\sigma}$ is a basis then $f_2 A_i = B_i f_1, i= 1, \cdots, w$, thus
$$ f = (f_1, f_2) \in {\rm Hom}_{\mathfrak{R}(K_w)^{gi}}(R_1, R_2). $$

Now we need to prove that ${\bf L}_{\boldsymbol{\sigma}}(f) = \phi$. Suppose ${\bf L}_{\boldsymbol{\sigma}}(f) = \overline{\phi}$ such that
$\overline{\phi} \pi_1 = \pi_2 f^{'}_2 = \phi \pi_1$. Then $(\overline{\phi} - \phi) \pi_1 = 0$ and
$\mathcal{C}_1 = {\rm im}\, \pi_1 \subset \ker(\overline{\phi} - \phi)$  therefore $\overline{\phi} = \phi$.

Finally, we show that
${\bf L}_{\boldsymbol{\sigma}}:{\rm Hom}(R_1, R_2) \rightarrow {\rm Hom}({\bf L}_{\boldsymbol{\sigma}}(R_1), {\bf L}_{\boldsymbol{\sigma}}(R_2))$ is injective. Let $f = (f_1, f_2)$, $g=(g_1, g_2) \in {\rm Hom}(R_1, R_2)$ be morphisms such that ${\bf L}_{\boldsymbol{\sigma}}(f) = \phi_1 = \phi_2 = {\bf L}_{\boldsymbol{\sigma}}(g)$, that is, $\phi_1 - \phi_2 = 0$.

\begin{eqnarray}\label{diagg2}
\xymatrix{ 0 \ar[r] & \E^a \ar[d]_{f_{1}^{'} - g^{'}_1} \ar[r]^{\alpha_1} & \F^b \ar[d]_{f_{2}^{'} - g_{2}^{'}} \ar[r]^{\pi_1}&  \mathcal{C}_1 \ar[r] \ar[d]^0 & 0\\
0 \ar[r] & \E^c \ar[r]^{\alpha_2} & \F^d \ar[r]^{\pi_2} & \mathcal{C}_2 \ar[r] & 0}
\end{eqnarray}

Given $\phi_1 - \phi_2 = 0 \in {\rm Hom}(\mathcal{C}_1, \mathcal{C}_2)$, doing the same construction as before,
$$ 0 \pi_1 = 0 \in {\rm Hom}(\F^b , \mathcal{C}_2) \simeq {\rm Hom}(\F^b,\F^d )$$
with isomorphism given by $\rho_2$ in (\ref{eqq}). Since

$$\rho_2 (f^{'}_2 - g^{'}_2) = \pi_2 \circ (f^{'}_2 - g^{'}_2) = 0$$ then $f^{'}_2 - g^{'}_2 = 0$ and so $f^{'}_2 = g^{'}_2$. Similarly, $0 \alpha_1 = 0 \in {\rm Hom}(\E^a, \F^d)$ and

$$\gamma_1 (f^{'}_1 - g^{'}_1) = \alpha_2 (f^{'}_1 - g^{'}_1) = 0 \alpha_1 = 0.$$ Since $\gamma_1$ injective, $f^{'}_1 - g^{'}_1 = 0$, then $f^{'}_1 = g^{'}_1.$  Therefore ${\bf L}_{\boldsymbol{\sigma}}$ is faithful.

\end{proof}

\begin{remark} \rm
Note that the functor ${\bf L}_{\boldsymbol{\sigma}}$ depends on the choice of the basis $\boldsymbol{\sigma}$. However let $\boldsymbol{\sigma}'$ be another basis for
${\rm Hom}(\E,\F)$. Let ${\bf L}_{\boldsymbol{\sigma}'}$ be the equivalence between the category of $(\E,\F,\boldsymbol{\sigma}')$-globally injective representations of $K_w$ and the cokernel bundles on $\pn$. Then if ${\bf G}$ is the inverse functor of ${\bf L}_{\boldsymbol{\sigma}'}$ we have that the functor ${\bf G}\circ {\bf L}_{\boldsymbol{\sigma}'}$ gives an equivalence between the categories $(\E,\F,\boldsymbol{\sigma})$- and $(\E,\F,\boldsymbol{\sigma}')$-globally injective representations of $K_w$.
\end{remark}

\begin{lemma}\label{decomp1}
For any choice of basis $\boldsymbol{\sigma}$, the functor
${\bf L}_{\boldsymbol{\sigma}}:\mathfrak{R}(K_{w})^{gi}\rightarrow\mathfrak{C}_X(\E,\F)$ defined above is additive and exact. In particular, if $R \simeq R_1 \oplus R_2$ is a globally injective representation, then
${\bf L}_{\boldsymbol{\sigma}}(R) \simeq {\bf L}_{\boldsymbol{\sigma}}(R_1) \oplus {\bf L}_{\boldsymbol{\sigma}}(R_2)$.
\end{lemma}

\begin{proof}
Checking the additivity of ${\bf L}_{\boldsymbol{\sigma}}$ is a simple exercise. We show its exactness in detail.

Let us prove that ${\bf L}_{\boldsymbol{\sigma}}$ preserves exact sequences. Let
$R_1 = (\{\kappa^{a_1}, \kappa^{b_1}\},$ \linebreak $\{A_i\}_{i=1}^w)$, $R_2 = (\{\kappa^{a_2}, \kappa^{b_2}\}, \{B_i\}_{i=1}^w)$ and $R_3 = (\{\kappa^{a_3}, \kappa^{b_3}\},\{C_i\}_{i=1}^w )$ be globally injective representations of $K_{w}$ and let $f: R_1 \rightarrow R_2$ and $g: R_2 \rightarrow R_3$  be morphisms such that the sequence 
$$ \xymatrix{
0 \ar[r] & R_1 \ar[r]^{f} & R_2 \ar[r]^{g} & R_3 \ar[r] & 0
}$$
is exact. We want to prove that
$$\xymatrix{
0 \ar[r] & \mathcal{C}_1 \ar[r]^{\varphi} & \mathcal{C}_2\ar[r]^{\psi} & \mathcal{C}_3 \ar[r] & 0
}$$
is also exact, where $\mathcal{C}_i = {\bf L}_{\boldsymbol{\sigma}}(R_i), i=1,2,3$ and
$\varphi = {\bf L}_{\boldsymbol{\sigma}}(f), \psi= {\bf L}_{\boldsymbol{\sigma}}(g)$. From the exact sequence of representations we get
$$\xymatrix{
& 0 \ar[d] & 0 \ar[d]& 0 \ar@{-->}[d] & \\
0 \ar[r] & \E^{a_1} \ar[r]^{\alpha_1} \ar[d]^{\mathds{1}_{\E}\otimes f_1} & \F^{ b_1}\ar[r]^{\pi_1} \ar[d]^{\mathds{1}_{\F}\otimes f_2} & \mathcal{C}_1 \ar[d]^{\varphi} \ar[r] & 0\\
0 \ar[r] & \E^{a_2}\ar[r]^{\alpha_2}\ar[d]^{\mathds{1}_{\E}\otimes g_1} & \F^{b_2} \ar[r]^{\pi_2} \ar[d]^{\mathds{1}_{\F}\otimes g_2} & \mathcal{C}_2 \ar[r] \ar[d]^{\psi} & 0\\
0 \ar[r] & \E^{a_3} \ar[r]^{\alpha_3} \ar[d] & \F^{b_3} \ar[r]^{\pi_3} \ar[d] & \mathcal{C}_3 \ar[r] \ar@{-->}[d] & 0\\
& 0 & 0 & 0 & \\
}$$
We need to show that $\varphi$ is injective and $\psi$ is surjective.

\begin{itemize}

\item $\psi$ is surjective:

It follows from the fact that $\pi_3 (\mathds{1}_{\F} \otimes g_2)$ is surjective.

\item $\varphi$ is injective.

Let us suppose $\varphi(s) = 0, s \in \mathcal{C}_1$. Then $s = \pi_1(v), v \in \F^{b_1}$ and
$$ 0 = \varphi \pi_1 (v) = \pi_2 (\mathds{1}_{\F}\otimes f_2)(v).$$
Since $\ker \pi_2 = \rm im \alpha_2$, there is $u \in \E^{ a_2}$ such that
\begin{equation}\label{Eq1}
(\mathds{1}_{\F }\otimes f_2)(v) = \alpha_2 (u)
\end{equation}

\end{itemize}

Note that
$$ \alpha_3 ( \mathds{1}_{\E}\otimes g_1) (u) = (\mathds{1}_{\F}\otimes g_2)(\alpha_2)(u) = (\mathds{1}_{\F}\otimes g_2)(\mathds{1}_{\F}\otimes f_2)(v) = 0 $$
and since $\alpha_3$ is injective, $(\mathds{1}_{\E}\otimes g_1)(u) = 0$ so $u = (\mathds{1}_{\E}\otimes f_1)(u')$ with $u' \in \E^{a_1}$.
We have
$$ \alpha_2 (u) = \alpha_2 (\mathds{1}_{\E}\otimes f_1)(u') = (\mathds{1}_{\F}\otimes f_2) \alpha_1 (u'). $$
From (\ref{Eq1}) we have $(\mathds{1}_{\F}\otimes f_2)(v) =(\mathds{1}_{\F}\otimes f_2)( \alpha_1 (u'))$. Since $(\mathds{1}_{\F}\otimes f_2)$ is injective, it follows that $v = \alpha_1(u')$ therefore
$$ s = \pi_1(v) = \pi_1 \alpha_1 (u') = 0.$$

Now suppose $R \simeq R_1 \oplus R_2$. Let us prove that ${\bf L}_{\boldsymbol{\sigma}}(R_1 \oplus R_2) \simeq {\bf L}_{\boldsymbol{\sigma}}(R_1) \oplus {\bf L}_{\boldsymbol{\sigma}}(R_2)$. We have the short exact sequence

\begin{eqnarray*} \xymatrix{
0 \ar[r] & R_1 \ar[r]^<<<<<{i_{R_1}} & R_1 \oplus R_2 \ar[r]^>>>>{ \pi_{R_2}} & R_2 \ar@<.3cm>[l]^>>>>{i_{R_2}} \ar[r] & 0  }
\end{eqnarray*}
where $i_{R_j}$ is the inclusion and $\pi_{R_j}$ the projection, $j=1,2$. Since the sequence above is split, $\pi_{R_2} \circ i_{R_2} = \mathds{1}_{R_2}$. Now since ${\bf L}_{\boldsymbol{\sigma}}$ is an exact functor, we have
\begin{eqnarray}\label{seqsplit} \xymatrix{
0 \ar[r] & {\bf L}_{\boldsymbol{\sigma}}(R_1) \ar[r]^<<<<{{\bf L}_{\boldsymbol{\sigma}}(i_{R_1})} & {\bf L}_{\boldsymbol{\sigma}}(R_1 \oplus R_2) \ar[r]^>>>>{ {\bf L}_{\boldsymbol{\sigma}}(\pi_{R_2})} & {\bf L}_{\boldsymbol{\sigma}}(R_2) \ar[r] \ar@<.3cm>[l]^>>>>{{\bf L}_{\boldsymbol{\sigma}}(i_{R_2})} & 0  }
\end{eqnarray}
Then
$$ {\bf L}_{\boldsymbol{\sigma}}(\pi_{R_2} \circ i_{R_2}) = {\bf L}_{\boldsymbol{\sigma}}(\pi_{R_2}) \circ {\bf L}_{\boldsymbol{\sigma}}(i_{R_2}) = {\bf L}_{\boldsymbol{\sigma}}(\mathds{1}_{R_2}) = \mathds{1}_{{\bf L}_{\boldsymbol{\sigma}}(R_2)} $$ therefore the sequence (\ref{seqsplit}) is split. Hence ${\bf L}_{\boldsymbol{\sigma}}(R_1 \oplus R_2) \simeq {\bf L}_{\boldsymbol{\sigma}}(R_1) \oplus {\bf L}_{\boldsymbol{\sigma}}(R_2)$.
\end{proof}

As an application of the previous results, we give a new, functorial proof for a result due to Brambilla, cf. \cite[Thm 4.3]{Bb}.

\begin{theorem}\label{decon}
Let $\mathcal{C}$ be a cokernel bundle of type $(\E,\F)$, given by the resolution
\begin{eqnarray}\label{cok1}
 \xymatrix{ 0  \ar[r] & \E^{ a} \ar[r]^{\alpha} & \F^{ b} \ar[r] & \mathcal{C} \ar[r]& 0 } ,\end{eqnarray}
and let $w=\dim{\rm Hom}(\E,\F)$.
\begin{itemize}
\item[(i)] If $\mathcal{C}$ is simple, then $a^2 + b^2 - wab \leq 1.$
\item[(ii)] If $a^2 + b^2 - wab \leq 1$, then there exists a non-empty open subset
$U\subset{\rm Hom}(\E^a,\F^b)$ such that for every $\alpha\in U$ the corresponding cokernel bundle is simple.
\end{itemize}
\end{theorem}

\begin{proof} 

To prove $(i)$, let $\mathcal{C}$ be a cokernel bundle given by resolution (\ref{cok1}) and suppose $\mathcal{C}$ is simple. By Theorem \ref{Teo1} there is a globally injective representation $R$ of $K_w$ such that $\mathcal{C} = {\bf L}_{\boldsymbol{\sigma}}(R)$. Since ${\bf L}_{\boldsymbol{\sigma}}$ is full, we have that $\kappa = {\rm Hom}(\mathcal{C},\mathcal{C}) \simeq {\rm Hom}(R,R)$, thus $R$ is simple and therefore, by Proposition \ref{Kac}, $q_w(a,b) = a^2+ b^2 -wab \leq 1.$

For the second claim, note that if $q_w(a,b) \leq 1$, there is a generic representation $R$ with dimension vector $(a,b)$ such that $R$ is Schur, by Proposition \ref{Kac}. Then there is a non-empty open subset 
$$ U \subset {\rm Hom}(\kappa^a, \kappa^b) \otimes \kappa^w \simeq {\rm Hom}(\E^a, \F^b) $$
such that every $R\in U$ is simple. Since 
${\rm Hom}(\mathcal{C},\mathcal{C}) \simeq {\rm Hom}(R,R) = \kappa$,
it follows that $\mathcal{C}$ is simple.
\end{proof}

The previous Theorem implies that if $a^2+ b^2 -wab > 1$ then $\mathcal{C}$ is not simple. However, more is true, and it is not difficult to establish the following stronger statement.

\begin{proposition}\label{decon1}
Under the same conditions as in Theorem \ref{decon}, if $a^2+ b^2 - w ab > 1$, then $\mathcal{C}$ is decomposable.
\end{proposition}

Under more restrictive conditions, Brambilla proved in \cite[Thm 6.3]{Bb} that if $\mathcal{C}$ is a \emph{generic} cokernel bundle such that $a^2+ b^2 - w ab > 1$, then $\mathcal{C} \simeq \mathcal{C}^n_k \oplus \mathcal{C}^m_{k+1}$, where $\mathcal{C}_k$ and $\mathcal{C}_{k+1}$ are {\it Fibonacci bundles}, $n,m \in \mathbb{N}$ (we refer to \cite{Bb} for the definition of Fibonacci bundles).

\begin{proof} 
Let $\mathcal{C}$ be any cokernel bundle given by the exact sequence (\ref{cok1}), such that
$a^2+ b^2 - w ab > 1$. Then there is a globally injective representation $R$ of $K_w$, such that $\mathcal{C} = {\bf L}_{\boldsymbol{\sigma}}(R)$ with dimension vector $(a,b)$ satisfying and
$q_w(a,b) = a^2+ b^2 - w ab > 1$. By Lemma \ref{Kac}, $R$ is decomposable. Then by Lemma \ref{decomp1}, $\mathcal{C}$ is also decomposable.
\end{proof}

Next, recall that a vector bundle $\E$ on $X$ is {\it exceptional} if it is simple and
${\rm Ext}^p(\E,\E) = 0$ for $p \geq 1$.

\begin{proposition}\label{excepcional}
Under the same conditions as in Theorem \ref{decon}, if $\mathcal{C}$ is exceptional, then
$a^2+ b^2 - w ab = 1$.
\end{proposition}
\begin{proof}
Since the functor ${\bf L}_{\boldsymbol{\sigma}}$ is exact, we have an isomorphism
$$ {\rm Ext}^1(R,R) \simeq 
{\rm Ext}^1({\bf L}_{\boldsymbol{\sigma}}(R), {\bf L}_{\boldsymbol{\sigma}}(R)). $$
Now we know from \cite{S} that
\begin{equation}\label{rmk}
q_w(a,b) = \dim {\rm Hom}(R,R) - \dim {\rm Ext}^1(R,R)
\end{equation}
hence if $\mathcal{C}$ is an exceptional cokernel bundle, then $q_w(a,b)= a^2 + b^2 - w a b=1$.
\end{proof}

However, the converse of the Proposition \ref{excepcional} is not true. For instance, consider the generic cokernel bundle given by the exact sequence
$$ \xymatrix{ 0 \ar[r] & \mathcal{O}_{\mathbb{P}^3} \ar[r] & \mathcal{O}_{\mathbb{P}^3}(4)^{35} \ar[r] & \mathcal{C} \ar[r] & 0 }. $$
We have $q_{35}(1,35) = 1$, but from the long exact sequence of cohomologies,
${\rm Ext}^2(\mathcal{C},\mathcal{C}) \simeq \kappa^{35}$ hence $\mathcal{C}$ is not exceptional. In order to establish the converse statement, we need stronger assumption, provided by Steiner bundles.

%%%%%%%%%%%%%%%%%%%%%%%%%%%%%%%%%%%%%%%%%%%%%%%%%%%%%%%%%%%%%%%%%%%%%5

\subsection{Steiner bundles}

Note that the Steiner bundles of type $(\E, \F)$ satisfying $\dim {\rm Hom}(\E,\F) \geq 3$, are a particular case of cokernel bundles, therefore all results in the previous section also hold for such Steiner bundles. Furthermore, the additional hypotheses satisfied by the sheaves $\E$ and $\F$ allow to establish the converse of Lemma \ref{decomp1} and Proposition \ref{excepcional}.  

Let us first consider the converse of Lemma \ref{decomp1}; more precisely, we prove the following statement.

\begin{theorem}
Let $X$ be a nonsingular projective variety of dimension $n$, and let
${\bf B} = (\boldsymbol{\mathcal{F}}_0, \cdots, \boldsymbol{\mathcal{F}}_n)$ be an $n$-block collection generating $D^b(X)$.
A Steiner bundle of type $(\F^i_{i_0},\F^j_{j_0})$ such that $w = \dim {\rm Hom}(\F^i_{i_0},\F^j_{j_0}) \geq 3$ is decomposable if and only if, for any choice of basis $\boldsymbol{\gamma}$ for
${\rm Hom}(\F^i_{i_0},\F^j_{j_0})$, the corresponding $(\F^i_{i_0}, \F^{j}_{j_0}, \boldsymbol \gamma)$-globally injective representation of $K_w$ is also decomposable.
\end{theorem}

The theorem follows easily from Lemma \ref{decomp1} and the following claim. Let
$\mathfrak{S}_{\F^i_{i_0}, \F^j_{j_0}}(X)$ denote the category of Steiner bundles of type $(\F^i_{i_0}, \F^j_{j_0})$ over $X$.

\begin{proposition}\label{dirsumsteiner}
The category $\mathfrak{S}_{\F^i_{i_0}, \F^j_{j_0}}(X)$ is closed under direct summands.
\end{proposition} 

\begin{proof}
Let ${}^{\vee}{\bf B}= (\boldsymbol{\mathcal{H}}_0, \cdots, \boldsymbol{\mathcal{H}}_n)$ where $\boldsymbol{\mathcal{H}}_i = ({\mathcal{H}}^i_1, \cdots, {\mathcal{H}}^i_{u_i})$, be the $n$-block collection which is left dual to ${\bf B}$, and let $\mathcal{S}$ be a Steiner bundle of type $(\F^i_{i_0},\F^j_{j_0})$ given by the short exact sequence
$$ \xymatrix{0 \ar[r] & (\F^i_{i_0})^a \ar[r] & (\F^j_{j_0})^b \ar[r] & \mathcal{S} \ar[r] & 0 }, $$
where $\F^i_{i_0}$ and $\F^j_{j_0}$ are elements of blocks $\boldsymbol{\mathcal{F}}_i$ and $\boldsymbol{\mathcal{F}}_j$ respectively, $0 \leq i < j \leq n$. 

If $\mathcal{S} \simeq \mathcal{S}_1 \oplus \mathcal{S}_2$,\, $0 \neq \mathcal{S}_i \subsetneq \mathcal{S},$ $i=1,2$, then we have that
$${\rm Ext}^{p}({\mathcal{H}}^m_q, \mathcal{S}) \simeq
{\rm Ext}^{p}(\mathcal{H}^{m}_q, \mathcal{S}_1) \oplus (\mathcal{H}^{m}_q, \mathcal{S}_2) .$$
It follows that ${\rm Ext}^p(\mathcal{H}^{m}_q, \mathcal{S}_l)$, $l=1,2$, vanish except for
$$ {\rm Ext}^{n-i-1}(\mathcal{H}^{n-i}_{i_0}, \mathcal{S}_l)= a_l, \, a_l \geq 0, \, l=1,2 $$
and
$$ {\rm Ext}^{n-j}(\mathcal{H}^{n-j}_{j_0}, \mathcal{S}_l)= b_l, \, b_l \geq 0, \, l=1,2 $$
with $a_1+a_2 = a$ and $b_1+b_2 = b$. Then from the cohomological characterisation, Theorem \ref{carmon}, one of the following possibilities must hold.

\begin{enumerate}
\item For $a_l \neq 0$ and $b_l \neq 0$, $l=1,2$, the bundles $\mathcal{S}_l$ are Steiner bundles given by 
$$ \xymatrix{
0 \ar[r] & (\F^i_{i_0})^{a_l} \ar[r] & (\F^j_{j_0})^{b_l} \ar[r] & \mathcal{S}_l \ar[r] & 0
} . $$

\item For $a_1,b_1,b_2 \neq 0$ and $a_2 = 0$, we have
$$ \xymatrix{
0 \ar[r] & (\F^i_{i_0})^{a_1} \ar[r] & (\F^j_{j_0})^{b_1} \ar[r] & \mathcal{S}_1 \ar[r] & 0} \,\,
{\rm and}\,\, \mathcal{S}_2 \simeq (\F_{j_0}^j)^{b_2}. $$

\item For $a_1 = 0$ and $b_1, a_2, b_2 \neq 0$, we have
$$ \mathcal{S}_1 \simeq (\F^i_{i_0})^{a_1} \,\, {\rm and}\,\,
\xymatrix{
0 \ar[r] & (\F^i_{i_0})^{a_2} \ar[r] & (\F^j_{j_0})^{b_2} \ar[r] & \mathcal{S}_2 \ar[r] & 0
} .$$
\end{enumerate}
\end{proof}

To complete this section, we consider the converse of Proposition \ref{excepcional}.

\begin{proposition}\label{corSteiner}
Let $\mathcal{S}$ be a Steiner bundle of type $(\E,\F)$ such that 
$$ w := \dim{\rm Hom}(\E,\F)\ge3, $$ 
given by the short exact sequence:
\begin{equation}\label{steiner}
\xymatrix{0 \ar[r] & \E^a \ar[r]^{\alpha} & \F^{b} \ar[r] & \mathcal{S} \ar[r] & 0}.
\end{equation}
\begin{itemize}
\item[$(i)$] If $\mathcal{S}$ is exceptional then $a^2 + b^2 - w a b = 1.$
\item[$(ii)$] If $a^2 + b^2 - w a b = 1 $ then there is a non-empty open subset $U \subset {\rm Hom}(\E^a, \F^b)$ such that for every $\alpha \in U$ the corresponding bundle $\mathcal{S}$ is exceptional.
\end{itemize}
\end{proposition}

\begin{proof} 
The first claim is just Proposition (\ref{excepcional}). For the second statement, we first show that 
if $\mathcal{S}_1, \mathcal{S}_2$ are Steiner bundles of type $(\E, \F)$, then
${\rm Ext}^p(\mathcal{S}_1, \mathcal{S}_2) = 0$ for $p \geq 2$.

Indeed, suppose $\mathcal{S}_i$, $i=1,2$, are given by short exact sequences
\begin{eqnarray}\label{eqsteiner}
\xymatrix{0 \ar[r] & \E^{a_i} \ar[r] & \F^{b_i} \ar[r] & \mathcal{S}_i \ar[r] & 0}
\end{eqnarray}
Applying the functor ${\rm Hom}( - ,\F)$ to the sequence ($\ref{eqsteiner}$) for $i=1$, we have
${\rm Ext}^p ( \mathcal{S}_1, \F) = 0$, $p \geq 2$. Applying ${\rm Hom}( - , \E)$ to the same sequence, we obtain ${\rm Ext}^q(\mathcal{S}_1, \E) = 0, q \geq 0$. Finally applying the functor
${\rm Hom}(\mathcal{S}_1, - )$ to the sequence ($\ref{eqsteiner}$), $i=2$, we conclude that
${\rm Ext}^j(\mathcal{S}_1, \mathcal{S}_2) = 0$ for $j \geq 2.$

Now to prove the second claim, start by supposing that $q_w(a,b) = a^2 + b^2 - w a b = 1$. By Theorem \ref{decon} item $(ii)$ there exists a non-empty open subset $U \subset {\rm Hom}(\E^a, \F^b)$ such that for every $\alpha  \in U$ the associated bundle $\mathcal{S}$ is simple. From (\ref{rmk}) we see that ${\rm Ext}^1(\mathcal{S},\mathcal{S}) = 0$. Finally, from the considerations above, we have
${\rm Ext}^{p}(\mathcal{S},\mathcal{S}) = 0$ for $p \geq 2$. Hence $\mathcal{S}$ is exceptional.
\end{proof}

\begin{remark}\rm
Soares also proved in \cite[Thm 2.2.7]{Soa}, using a different method, that a generic Steiner bundle of type $(\E,\F)$ given by the short exact sequence (\ref{steiner}) is exceptional if and only if $a^2 + b^2 - w a b = 1$.
\end{remark}

%%%%%%%%%%%%%%%%%%%%%%%%%%%%%%%%%%%% SYZYGY BUNDLES %%%%%%%%%%%%%%%%%%%%%%%%%%%%%%%%%%%%%%%%%%%

\section{Syzygy bundles and quivers}

In this section we relate a different class of vector bundles, the syzygy bundles, with representations of quivers. A locally free sheaf $\mathcal{G} $ given by the short exact sequence
\begin{eqnarray}\label{syz}
\xymatrix{0 \ar[r] & \mathcal{G} \ar[r] & {\opn(-d_1)^{a_1} \oplus \cdots \oplus \opn(-d_m)^{a_m}}  \ar[r]^>>>> \alpha & {\opn^c}\ar[r] & 0 }
\end{eqnarray}
is called a \emph{syzygy bundle}. Here, $\alpha = (\alpha_1,\alpha_2, \ldots, \alpha_m)$ is a surjective map of sheaves on $\mathbb{P}^{n}$ given by
$$ \alpha(f_1,f_2, \ldots, f_m) = \sum_{i=1}^m \alpha_i f_i  $$
where $f_1, \ldots, f_m$ are homogeneous polynomials of degree $d_1, \ldots,d_m$ in \linebreak
$\kappa[X_0, \ldots, X_n]$ and $d_i$ are distinct positive integers. Let us assume $0 \le d_m < \cdots < d_1.$

Note that for $m=1$, the dual bundle $\mathcal{G}^{*}$ is a cokernel bundle. However, the same is not true for $m>1$, since the bundle $\F = \opn(d_1)^{a_1} \oplus \cdots \oplus \opn(d_m)^{a_m}$ is not simple.

To relate syzygy bundles with representations of quivers, we restrict ourselves, for the sake of simplicity, to the case $m=2$. The results for the general case are the same, but the notation becomes more complicated. Thus we set $m = 2$, and consider exact sequences of the form 
\begin{eqnarray}\label{syzy}
\xymatrix{0 \ar[r] & \mathcal{G} \ar[r] & {\opn(-d_1)^{a} \oplus \opn(-d_2)^{b}}  \ar[r]^>>>>>{\alpha_1,\alpha_2} & {\opn^c}\ar[r] & 0 }
\end{eqnarray}
with $d_1>d_2$. We denote by $\mathfrak{S}yz(d_1,d_2)$ the category of syzygy bundles given by short exact sequences as in (\ref{syzy}) above.

Fix, for $i=1,2$, a basis $\boldsymbol \sigma_i = \{ f^i_1, \ldots ,f^i_{w_i} \}$ of $H^0(\opn(d_i))$, where 
$w_i =\binom{n + d_i} {d_i}$. Consider the quiver below, which will be denoted by $A_{w_1,w_2}$:
\begin{eqnarray}\label{quivsyz}
\xymatrix{\bullet \ar@<1.8ex>[r]^{1} \ar@<-1.8ex>[r]^{\vdots}_{w_1} & \bullet & \ar@<1.8ex>[l]_{\vdots}^{w_2} \ar@<-1.8ex>[l]_{1}
\bullet }
\end{eqnarray}
If $(a,b,c)$ is a dimension vector of this quiver, its Tits form is given by
\begin{equation} \label{tits2}
q_{w_1,w_2}(a,b,c) = a^2+ b^2 + c^2 - w_1 a b - w_2 b c. 
\end{equation}

Let $R = (\{\kappa^a, \kappa^b, \kappa^c\}, \{A_i\}_1^{w_1}, \{B_j\}_1^{w_2})$ be a representation of $A_{w_1,w_2}$, where each $A_i$ is a $c \times a$ matrix, and each $B_j$ is a $c \times b$ matrix with entries in $\kappa$. We define
$$ \alpha_1 = \sum_{i=1}^{w_1} A_i \otimes f^1_i \,\, \text{and} \,\,
\alpha_2 = \sum_{j=1}^{w_2} B_j \otimes f^2_j , $$
so that we have a map 
\begin{equation} \label{alpha12}
(\alpha_1,\alpha_2):\opn(-d_1)^a \oplus \opn(-d_2)^b \rightarrow \opn^c.
\end{equation}

\begin{definition}
A representation $R$ of $A_{w_1,w_2}$ is $(\boldsymbol\sigma_1,\boldsymbol\sigma_2)$-\emph{globally surjective} if the map $(\alpha_1, \alpha_2)$ is surjective.
\end{definition}

Denote by $\mathfrak{R}(A_{w_1,w_2})^{gs}$ the category of $(\boldsymbol\sigma_1,\boldsymbol\sigma_2)$-globally surjective representations of $A_{w_1,w_2}$. We will now build a functor ${\bf G}_{\boldsymbol\sigma_1, \boldsymbol \sigma_2}$ between the $\mathfrak{R}(A_{w_1,w_2})^{gs}$ and the category of syzygy bundles $\mathfrak{S}yz(d_1,d_2)$.

First, let $R = (\{\kappa^a, \kappa^b, \kappa^c\}, \{A_i\}_{i=1}^{w_1}, \{B_j\}_{i=1}^{w_2})$ be a globally surjective representation of $A_{w_1,w_2}$. We define the sheaf 
$$ {\bf G}_{\boldsymbol\sigma_1, \boldsymbol\sigma_2}(R) := \ker (\alpha_1,\alpha_2), $$
where $(\alpha_1, \alpha_2)$ is the map defined above in (\ref{alpha12}). Note that, since $R$ is globally surjective, ${\bf G}_{\boldsymbol\sigma_1, \boldsymbol\sigma_2}(R)$ is a vector bundle, and it is given by the exact sequence (\ref{syzy}).

Now let $\{g_1,g_2, h\}$ be a morphism between the globally surjective representations
$$ R = (\{\kappa^a, \kappa^b, \kappa^c\}, \{A_i\}_{i=1}^{w_1}, \{B_j\}_{i=1}^{w_2}) ~~ {\rm and} $$
$$ R' = (\{\kappa^{a'}, \kappa^{b'}, \kappa^{c'}\},\{A'_i\}_{i=1}^{w_1}, \{B'_j\}_{i=1}^{w_2}). $$
The following diagram commutes for $i = 1, \ldots, w_1$ and $j = 1, \ldots, w_2$.
\begin{eqnarray}\label{diag}
\xymatrix{
\kappa^a \ar@/^0.2cm/[rr]^{g_1} \ar[dr]^{A_i} && \kappa^{a'} \ar[dr]^{A'_i} & \\
& \kappa^c \ar[rr]^h && \kappa^{c'} \\
\kappa^b \ar[ur]_{B_i} \ar@/_0.2cm/[rr]_{g_2} && \kappa^{b'} \ar[ur]_{B'_j} &}
\end{eqnarray}
It induces the following diagram:
\begin{eqnarray}\label{disyz}
\xymatrix{ 0 \ar[r] & \mathcal{G}\ar@{-->}[d]^{\phi} \ar[r]^<<<<{i_1}& \opn(-d_1)^a \oplus \opn(-d_2)^b \ar[d]^{M} \ar[r]^>>>>>{\alpha_1, \alpha_2}& \opn^c \ar[d]^{h \otimes \mathds{1}_{\opn}} \ar[r] &0  \\
0 \ar[r] & \mathcal{G}' \ar[r]^<<<<{i_2}& \opn(-d_1)^{a'} \oplus \opn(-d_2)^{b'} \ar[r]^>>>>{\alpha'_1, \alpha'_2}& \opn^{c'} \ar[r] &0 }
\end{eqnarray}
where
$$ M = \left(
\begin{array}{cc}
g_1 \otimes \mathds{1}_{\opn(-d_1)} & 0\\
0 & g_2 \otimes \mathds{1}_{\opn(-d_2)}\\
\end{array} \right)
$$
The commutativity of (\ref{diag}) implies the commutativity of the right square in (\ref{disyz}). We then have an induced  morphism $\phi:\mathcal{G}={\bf G}_{\boldsymbol{\sigma}_1, \boldsymbol{\sigma}_2}(R)\to
\mathcal{G}'={\bf G}_{\boldsymbol{\sigma}_1, \boldsymbol{\sigma}_2}(R')$, which we define to be
${\bf G}_{\boldsymbol{\sigma}_1, \boldsymbol{\sigma}_2}(g_1,g_2,h)$.

%The proof of the next lemma is a simple exercise. 

\begin{lemma}
The functor ${\bf G}_{\boldsymbol{\sigma}_1,\boldsymbol{\sigma}_2}$ is faithful and essentially surjective.
\end{lemma}
\begin{proof}
We prove that ${\rm Hom}(R,R') \rightarrow {\rm Hom}({\bf G}(R), {\bf G}(R'))$ is injective. Let \linebreak $\{g_1,g_2,h\}$ be a morphism between $R$ and $R'$ such that ${\bf G}(\{g_1,g_2,h\})=0$, that is, $\phi = 0$.
Since the diagram (\ref{disyz}) commutes if $\phi = 0$ then $g_1 = g_2 = h = 0$, hence ${\bf G}$ is faithful.

Let $\mathcal{G}$ be a syzygy bundle with resolution
$$ \xymatrix{ 0 \ar[r] & \mathcal{G} \ar[r] & {\opn(-d_1)^{a} \oplus \opn(-d_2)^{b}}  \ar[r]^>>>>>{\alpha_1,\alpha_2} & {\opn^c}\ar[r] & 0 }$$ 
Then the maps $\alpha_1$ and $\alpha_2$ are given by
$$\alpha_1 = \sum_{i=1}^{w_1} A_i \otimes f_i^1 \,\, \text{and} 
\,\, \alpha_2 = \sum_{j=1}^{w_2} B_j \otimes f^2_j $$ 
with $A_i \in {\rm Hom}(k^a, k^c)$ and  $B_j \in {\rm Hom}(k^b, k^c)$. Therefore
$$ R = (\{k^a, k^b, k^c\}, \{A_i\}_1^{w_1}, \{B_j\}_1^{w_2}) $$ 
is a globally surjective representation of (\ref{quivsyz}) such that ${\bf G}(R) = \mathcal{G}$.
\end{proof}

\begin{remark}\rm
Note that ${\bf G}_{\boldsymbol{\sigma}_1,\boldsymbol{\sigma}_2}$ is not full, since not every 
$$ M \in {\rm Hom}\,(\opn(-d_1)^{a}\oplus \opn(-d_2)^{b}, \opn(-d_1)^{a'}\oplus \opn(-d_2)^{b'}) $$
is necessarily diagonal. Hence, the categories $\mathfrak{R}(A_{w_1,w_2})^{gs}$ and $\mathfrak{S}yz(d_1,d_2)$ are not, in general, equivalent.
\end{remark}

This completes the proof of the first part of Theorem \ref{thmsyzygy}. To establish its second part, we first need the following two lemmas.

\begin{lemma}\label{quoc}
The category of globally surjective representations of the quiver $A_{w_1,w_2}$ is closed under quotients, and hence closed under direct summands.
\end{lemma}

\begin{proof}
Let 
$$ R = (\{\kappa^a, \kappa^b, \kappa^c\}, \{A_i\}_{i=1}^{w_1}, \{B_j\}_{j=1}^{w_2}) $$
be a ($\boldsymbol{\sigma}_1,\boldsymbol{\sigma}_2$)-globally surjective representation of $A_{w_1,w_2}$ and
$$ R' = (\{\kappa^{a'}, \kappa^{b'}, \kappa^{c'}\}, \{A'_i\}_{i=1}^{w_1}, \{B'_j\}_{j=1}^{w_2}) $$ 
be a subrepresentation of $R$. We want to prove that the quotient representation
$$ R/R' = (\{\kappa^a/\kappa^{a'}, \kappa^b/\kappa^{b'}, \kappa^c/\kappa^{c'}\}, \{C_i\}_{i=1}^{w_1}, \{D_j\}_{j=1}^{w_2}), $$
where $C_i$ and $D_j$ are the maps induced by $A_i$ and $B_j$ respectively, is also globally surjective. We have the diagram

$$\xymatrix{ 0 \ar[d]  &  0 \ar[d] & 0 \ar[d]   \\
\kappa^{a'}\ar[dd]_{l_1}\ar@<1.8ex>[r]^{A'_1} \ar@<-1.8ex>[r]^{\vdots}_{A'_{w_1}} & \kappa^{c'} \ar[dd]_{l_3} & \kappa^{b'} \ar[dd]_{l_2} \ar@<1.8ex>[l]^{B'_{w_2}} \ar@<-1.8ex>[l]^{\vdots}_{B'_{1}} \\
& & &
\\
  \kappa^{a}\ar[dd]_{p_1} \ar@<1.8ex>[r]^{A_1} \ar@<-1.8ex>[r]^{\vdots}_{A_{w_1}} & \kappa^{c} \ar[dd]_{p_3} & \kappa^{b} \ar[dd]_{p_2} \ar@<1.8ex>[l]^{B_{w_2}} \ar@<-1.8ex>[l]^{\vdots}_{B_{1}} \\
& & &
\\
  {\kappa^a/\kappa^{a'}}\ar[d] \ar@<1.8ex>[r]^{C_1} \ar@<-1.8ex>[r]^{\vdots}_{C_{w_1}} & {\kappa^c/\kappa^{c'}} \ar[d] & {\kappa^b/\kappa^{b'}} \ar[d] \ar@<1.8ex>[l]^{D_{w_2}} \ar@<-1.8ex>[l]^{\vdots}_{D_{1}} \\
0 & 0 & 0
}$$
where $l_i$ are the inclusions and $p_i$ the projections $i=1,2,3$. Now consider the commutative diagram
$$\xymatrix{
\opn(-d_1)^{a}\oplus \opn(-d_2)^{b} \ar[d]^{M} \ar[r]^>>>>>>{(\alpha_1, \alpha_2)} &
\opn^c \ar[d]^{p_3 \otimes \mathds{1}_{\opn}} \\
\opn(-d_1)^{(a-a')}\oplus \opn(-d_2)^{(b-b')} \ar[r]^>>>>{(\gamma_1, \gamma_2)} & \opn^{(c-c')}
}$$
where
$$ M = \left(
\begin{array}{cc}
p_1 \otimes \mathds{1}_{\opn(-d_1)} & 0\\
0 & p_2 \otimes \mathds{1}_{\opn(-d_2)}\\
\end{array} \right)
$$
and $\gamma_1 = \sum_{i=1}^{w_1} C_i \otimes f^1_i$, $\gamma_2 = \sum_{j=1}^{w_2} D_j \otimes f^2_j$.

Since $p_i$ is surjective and $(\alpha_1, \alpha_2)$ is surjective for every $P \in \mathbb{P}^n$, we have that the map $(\gamma_1, \gamma_2)$ is also surjective for every point $P \in \mathbb{P}^n$, hence the quotient representation $R/R'$ is globally surjective.
\end{proof}

\begin{lemma}\label{decompsyz}
Let $R$ be a decomposable globally surjective representation of $A_{w_1,w_2}$. Then
${\bf G}_{\boldsymbol{\sigma}_1,\boldsymbol{\sigma}_2}(R)$ is also decomposable.
\end{lemma}
\begin{proof}
Let $R \simeq R_1 \oplus R_2$ be a decomposable globally surjective representation. From Lemma \ref{quoc} we have that $R_1$ and $R_2$ are globally surjective. Let
$\mathcal{G}_i = {\bf G}_{\boldsymbol{\sigma}_1, \boldsymbol{\sigma}_2}(R_i)$, $i = 1, 2$ be given by the short exact sequence
$$\xymatrix{0 \ar[r] & \mathcal{G}_i \ar[r] & {\opn(-d_1)^{a_i} \oplus \opn(-d_2)^{b_i}}  \ar[r]^>>>>>{\alpha^i} & {\opn^{c_i}}\ar[r] & 0 }$$
where $\alpha_i = (\alpha^i_1, \alpha^i_2)$, $i=1,2$. Since
\begin{eqnarray*} 
\mathcal{G} = {\bf G}_{\boldsymbol{\sigma}_1,\boldsymbol{\sigma}_2}(R) =
\ker (\alpha_1 \oplus \alpha_2) & \simeq & 
\ker \alpha_1 \oplus \ker \alpha_2 = \\
& = & {\bf G}_{\boldsymbol{\sigma}_1,\boldsymbol{\sigma}_2}(R_1) \oplus {\bf G}_{\boldsymbol{\sigma}_1,\boldsymbol{\sigma}_2}(R_2),
\end{eqnarray*}
it follows that $\mathcal{G}$ is decomposable.
\end{proof}

We are finally in position to complete the proof of Theorem \ref{thmsyzygy}. Indeed, fix bases $\boldsymbol{\sigma}_j$ for $H^0(\opn(d_j))$, $j=1,2$. For every syzygy bundle $\mathcal{G}$ given by a short exact sequence of the form (\ref{syzy}), one can find a $(\boldsymbol{\sigma}_1,\boldsymbol{\sigma}_2)$-globally surjective representation $R$ of $A_{w_1,w_2}$ with dimension vector $(a,b,c)$ with
${\bf G}_{\boldsymbol{\sigma}_1,\boldsymbol{\sigma}_2}(R)=\mathcal{G}$. 
If $q_{w_1,w_2}(a,b,c) > 1$, then  $R$ is decomposable, by Lemma \ref{Kac}, and it must decompose as a sum of $(\boldsymbol{\sigma}_1,\boldsymbol{\sigma}_2)$-globally surjective representations by Lemma \ref{quoc}. Therefore
Lemma \ref{decompsyz} implies that $\mathcal{G}$ is also decomposable. 

\begin{remark}\label{m>2}\rm
All the results can be generalized for syzygy bundles with $m \geq 2$. To build the associated quiver, we add a vertex to the quiver with $w_i = \dim H^{0}(\opn(d_i))$ arrows from this vertex to the vertex associated to $\opn^{\oplus c}$, for each term $\opn(-d_i)^{\oplus a_i}$.
\end{remark}

%%%%%%%%%%%%%%%%%%%%%%%%%%%%%%%%%%%%%%%%%%%%%%%%%%%%%%%%%%%%%%%%%%%%%%%%%%%%%%%%%%%%%%%%%%%%%%%%%%%%

\section{Monads and representations of quivers}\label{sec5}

Recall that a {\it monad} $M^{\bullet}$ on a projective variety $X$ is a complex of locally free sheaves
\begin{eqnarray}\label{mon}
M^{\bullet}:\xymatrix{ \mathcal{A}^{\oplus a} \ar[r]^{\alpha} & \mathcal{B}^{\oplus b} \ar[r]^{\beta} & \mathcal{C}^{\oplus c}  }
\end{eqnarray}
where $\alpha$ is injective and $\beta$ is surjective. The coherent sheaf $\mathcal{E}:= \ker \beta /  \im\alpha$ is called the {\it cohomology of} $M^{\bullet}$; note that $\mathcal{E}$ is locally free if and only if the map $\alpha_P$ on the fibers is injective for every point $P \in X$. 

Now let $m = \dim \rm{Hom}(\mathcal{A}, \mathcal{B})$ and
$n = \dim \rm{Hom}(\mathcal{B}, \mathcal{C})$. We also assume that $\mathcal{A}, \mathcal{B}, \mathcal{C}$ are simple vector bundles, and that the cohomology sheaf $\mathcal{E}$ is locally free.
We will denote the category of such monads by $\mathfrak{M}_{\mathcal{A,B,C}}$, regarding it as a full subcategory of the category of complexes of coherent sheaves on $X$.

Next, consider the quiver $K_{m,n}$ given by the graph
$$ \xymatrix{
\bullet \ar@<1.8ex>[r]^1 \ar@<-1.8ex>[r]^{\vdots}_m & \bullet \ar@<-1.8ex>[r]^{\vdots}_n \ar@<1.8ex>[r]^1 & \bullet}$$
The category of representations of $K_{m,n}$ is denoted by $\mathfrak{R}(K_{m,n})$. Note that its Tits form is given by
\begin{equation}\label{tits-monad}
q_{m,n}(a,b,c) = a^2 + b^2 + c^2 - mab - nbc.
\end{equation}

\subsection{Proof of Theorem \ref{thm-monads}} \label{pfthm1.3}

We begin by describing a functor from $\mathfrak{M}_{\mathcal{A,B,C}}$ to $\mathfrak{R}(K_{m,n})$ in a manner similar to what was done in the previous sections. Choose bases $\boldsymbol{\gamma} = \{\gamma_1, \cdots, \gamma_m\}$ of $\rm{Hom}(\mathcal{A}, \mathcal{B})$ and $\boldsymbol{\sigma} = \{\sigma_1, \cdots, \sigma_n\}$ of
$\rm{Hom}(\mathcal{B}, \mathcal{C})$. We can write
$$ \alpha = \sum_{i=1}^m A_i \otimes \gamma_i ~~{\rm and} ~~\beta = \sum_{j=1}^n B_j \otimes \sigma_j $$
where each $A_i$ is a $b\times a$ matrix with entries in $\kappa$, and each $B_j$ is a $c \times b$ matrix with entries in $\kappa$. 

Now let 
\begin{equation}\label{funtor G}
{\bf G}_{\boldsymbol{\gamma}, \boldsymbol{\sigma}}: \mathfrak{M}_{\mathcal{A,B,C}} \rightarrow \mathfrak{R}(K_{m,n})
\end{equation}
be the functor that to each monad $M^{\bullet}$ as in (\ref{mon}) with maps $\alpha$ and $\beta$, associates the representation
$R =(\{\kappa^a, \kappa^b, \kappa^c\}, \{A_i\}_{i=1}^m,  \{B_j\}_{j=1}^m)$. Let 
$\varphi_{\bullet} = (f,g,h)$ be a morphism between the monads $M^{\bullet}_1$ and $M^{\bullet}_2$ below 
$$ \xymatrix{\mathcal{A}^{a_1} \ar[d]_{f}\ar[r]^{\alpha_1} & \mathcal{B}^{b_1} \ar[d]_{g}\ar[r]^{\beta_1} & \mathcal{C}^{c_1} \ar[d]_{h}  \\
\mathcal{A}^{a_2} \ar[r]_{\alpha_2} & \mathcal{B}^{b_2} \ar[r]_{\beta_2} & \mathcal{C}^{c_2}
}$$

Since $\mathcal{A}$, $\mathcal{B}$ and $\mathcal{C}$ are simple, it follows that
$$ (f,g,h) =
(A \otimes \mathds{1}_{\mathcal{A}}, B \otimes \mathds{1}_{\mathcal{B}}, C \otimes \mathds{1}_{\mathcal{C}}) $$
where $A,B$ and $C$ are, respectively, $a_2\times a_1$, $b_2 \times b_1$ and $c_2 \times c_1$ matrices with entries in $\kappa$. If
$$ {\bf G}_{\boldsymbol{\gamma},\boldsymbol{\sigma}}(M^{\bullet}_1) =
(\{\kappa^{a_1}, \kappa^{b_1}, \kappa^{c_1}\}, \{A^1_i\}_{i=1}^{m}, \{B^1_j\}_{j=1}^{n}) $$ 
$$ {\rm and}~~ {\bf G}_{\boldsymbol{\gamma},\boldsymbol{\sigma}}(M^{\bullet}_2) =
(\{\kappa^{a_2}, \kappa^{b_2}, \kappa^{c_2}\}, \{A^2_i\}_{i=1}^{m}, \{B^2_j\}_{j=1}^{n}), $$
we then have 
\begin{eqnarray}\label{morphism}
\xymatrix{\kappa^{a_1}\ar[dd]_{A} \ar@<1.8ex>[r]^{A^1_1} \ar@<-1.8ex>[r]^{\vdots}_{A^1_m} & \kappa^{b_1} \ar[dd]_{B} \ar@<-1.8ex>[r]^{\vdots}_{B^1_n} \ar@<1.8ex>[r]^{B^1_1} & \kappa^{c_1}\ar[dd]_{C}\\
& & & \\
\kappa^{a_2} \ar@<1.8ex>[r]^{A^2_1} \ar@<-1.8ex>[r]^{\vdots}_{A^2_m} & \kappa^{b_2} \ar@<-1.8ex>[r]^{\vdots}_{B^2_n} \ar@<1.8ex>[r]^{B^2_1} & \kappa^{c_2}}
\end{eqnarray}
$$B A^1_i = A^2_i A \;\; {\rm and} \;\;  CB^1_j = B^2_j B \;\; {\rm for} \;\; i=1,\cdots,m \;\; \mbox{and} \;\; j=1,\cdots, n.$$
Hence the matrices $A$,$B$ and $C$ define a morphism between the representations. From the construction of the functor we see that
$$ {\bf G}_{\boldsymbol{\gamma},\boldsymbol{\sigma}}:
{\rm Hom}(M^{\bullet}_1, M^{\bullet}_2)\rightarrow
{\rm Hom}({\bf G}_{\boldsymbol{\gamma},\boldsymbol{\sigma}}(M^{\bullet}_1),
{\bf G}_{\boldsymbol{\gamma},\boldsymbol{\sigma}}(M^{\bullet}_2)) $$
is an isomorphism, thus we have the following result, which corresponds to the first part of Theorem \ref{thm-monads}

\begin{proposition}\label{equivmonads}
The category $\mathfrak{M}_{\mathcal{A,B,C}}$ is equivalent to a full subcategory of $\mathfrak{R}(K_{m,n})$. 
\end{proposition}

Let us further characterise the subcategory of $\mathfrak{R}(K_{m,n})$ obtained in this way. The monad conditions imply that $\alpha(P)$ is injective and $\beta(P)$ is surjective for every $P\in X$. Therefore we say that a representation 
$$ R = (\{\kappa^a, \kappa^b, \kappa^c\}, \{A_i\}_{i=1}^m,  \{B_j\}_{j=1}^n) $$
is $(\boldsymbol{\boldsymbol{\gamma}},\boldsymbol{\boldsymbol{\sigma}})$-globally injective and surjective if 
$\alpha(P) = \sum_{i=1}^m A_i \otimes\gamma_i(P)$ is injective and $\beta(P) = \sum_{j=1}^n B_j \otimes \sigma_j(P)$ is surjective, for every $P \in X$. In addition, the matrices $A_i$ and $B_j$ must satisfy quadratic equations imposed by the condition $\beta\alpha=0$: 
$$ \sum_{1 \leq i \leq j \leq m} (B_i A_j + B_j A_i)(\sigma_i \gamma_j) = 0 ;$$
note that the precise relation depends on the choice of bases $\boldsymbol{\gamma}$ and $\boldsymbol{\sigma}$. We denote by $\mathfrak{G}^{\rm{gis}}_{m,n}$ the full subcategory of $\mathfrak{R}(K_{m,n})$ consisting of the objects satisfying the conditions above. 

In order to prove the second part of Theorem \ref{thm-monads}, our first goal is to prove that $\mathfrak{G}^{\rm gis}_{m,n}$ is closed under direct summands. 

\begin{lemma}\label{subobrep}
The category $\mathfrak{G}^{\rm{gis}}_{m,n}$ is closed under direct summands.
\end{lemma}

\begin{proof}
It is a general fact that if $S$ is a subrepresentation of a quiver representation $R$ which satisfies the given relations, then $S$ also satisfies the same relations.

Moreover, every subrepresentation of a $\boldsymbol{\gamma}$-globally injective representation will also be $\boldsymbol{\gamma}$-globally injective (cf. Lemma \ref{subob} above), while any quotient representation of a $\boldsymbol{\sigma}$-globally surjective representation will also be $\boldsymbol{\sigma}$-globally surjective (cf. Lemma \ref{quoc} above).
\end{proof}

Next, the previous lemma allows us to relate the decomposability of the monad with the decomposability of the associated quiver representation.

\begin{proposition} \label{decomp-monad}
A monad $M^{\bullet}$ is decomposable if and only if the associated quiver representation ${\bf G}_{\boldsymbol{\gamma},\boldsymbol{\sigma}}(M^{\bullet})$ is decomposable. In addition, if ${\bf G}_{\boldsymbol{\gamma},\boldsymbol{\sigma}}(M^{\bullet})$ is decomposable, then the cohomology of $M^{\bullet}$ is a decomposable vector bundle.
\end{proposition}

\begin{proof}
We begin by showing that the functor
${\bf G}_{\boldsymbol{\gamma},\boldsymbol{\sigma}}:\mathfrak{M}_{\mathcal{A,B,C}} \rightarrow \mathfrak{G}^{gis}_{m,n}$ preserves direct sums, that is, ${\bf G}_{\boldsymbol{\gamma},\boldsymbol{\sigma}}(M^{\bullet}_1 \oplus M^{\bullet}_2) \simeq
{\bf G}_{\boldsymbol{\gamma},\boldsymbol{\sigma}}(M^{\bullet}_1) \oplus
{\bf G}_{\boldsymbol{\gamma},\boldsymbol{\sigma}}(M^{\bullet}_2)$. In particular, if $M^{\bullet}$ is decomposable then $R = {\bf G}_{\boldsymbol{\gamma},\boldsymbol{\sigma}}(M^{\bullet})$ is decomposable.

Indeed, consider a monad $M^{\bullet} = M^{\bullet}_1 \oplus M^{\bullet}_2$ given by
$$ \xymatrix{\mathcal{A}^{a_1+a_2} \ar[r]^{\alpha} & \mathcal{B}^{b_1 + b_2} \ar[r]^{\beta}&
\mathcal{C}^{c_1+ c_2}} $$
where $\alpha =  \alpha_1\oplus \alpha_2$  and $\beta = \beta_1 \oplus \beta_2$ with
$\alpha_i \in {\rm Hom}(\mathcal{A}^{a_i}, \mathcal{B}^{b_i})$ and
$ \beta_i \in {\rm Hom}(\mathcal{B}^{b_i}, \mathcal{C}^{c_i}),$ $i=1,2$. We write $\alpha_i, \beta_i$ as
$$ \alpha_i = \sum_{l=1}^m A^i_l\otimes \gamma_l \,\, {\mbox{and}}\,\,
\beta_i = \sum_{j=1}^n B^i_j \otimes \sigma_j ,\,\, i=1,2. $$
Then ${\bf G}_{\boldsymbol{\gamma}, \boldsymbol{\sigma}}(M^{\bullet}_1 \oplus M^{\bullet}_2)$ is the representation 
$$ \xymatrix{\kappa^{a_1 \oplus a_2}
 \ar@<1.8ex>[r]^{A^1_1\oplus A^2_1} \ar@<-1.8ex>[r]^{\vdots}_{A^1_m\oplus A^2_m} & \kappa^{b_1 \oplus b_2} \ar@<-1.8ex>[r]^{\vdots}_{B^1_n\oplus B^2_n} \ar@<1.8ex>[r]^{B^1_1\oplus B^2_n} & \kappa^{c_1\oplus c_2}} $$
and it is clear that 
$$ {\bf G}_{\boldsymbol{\gamma}, \boldsymbol{\sigma}}(M^{\bullet}_1\oplus M^{\bullet}_2) =
(\{\kappa^{a_1+ a_2}, \kappa^{b_1+ b_2}, \kappa^{c_1+ c_2} \}, \{A^1_i\oplus A^2_i\}_{i=1}^m, \{B^1_j\oplus B^2_j\}_{j=1}^n) $$
$$ \simeq (\{\kappa^{a_1}, \kappa^{b_1}, \kappa^{c_1}\}, \{A^1_i\}_{i=1}^{m},\{B^1_j\}_{j=1}^{n})\oplus
(\{\kappa^{a_2}, \kappa^{b_2}, \kappa^{c_2}\}, \{A^2_i\}_{i=1}^{m},\{B^2_j\}_{j=1}^n) $$ 
$$ =
{\bf G}_{\boldsymbol{\gamma}, \boldsymbol{\sigma}}((M^{\bullet}_1) \oplus {\bf G}_{\boldsymbol{\gamma}, \boldsymbol{\sigma}}((M^{\bullet}_2) .$$

For the converse, suppose $R = {\bf G}_{\boldsymbol{\gamma}, \boldsymbol{\sigma}}(M^{\bullet}) \simeq R_1 \oplus R_2$. By Lemma
\ref{subobrep} we know that there are monads $M^{\bullet}_i$, for $i=1,2$, such that 
$R_i = {\bf G}_{\boldsymbol{\gamma}, \boldsymbol{\sigma}}(M^{\bullet}_i)$. It follows that
$$ {\bf G}_{\boldsymbol{\gamma}, \boldsymbol{\sigma}}(M^{\bullet}) \simeq
{\bf G}_{\boldsymbol{\gamma}, \boldsymbol{\sigma}}(M^{\bullet}_1) \oplus {\bf G}_{\boldsymbol{\gamma}, \boldsymbol{\sigma}}(M^{\bullet}_2) \simeq
{\bf G}_{\boldsymbol{\gamma},\boldsymbol{\sigma}}(M^{\bullet}_1 \oplus M^{\bullet}_2) $$
hence $M^{\bullet}$ is decomposable.

The second claim follows easily from the observation that if a monad is decomposable, then so is its cohomology sheaf.
\end{proof}

The completion of the proof of Theorem \ref{thm-monads} is at hand: if $M^{\bullet}$ is a monad of the form (\ref{mon}) with $(a,b,c)$ satisfying $q_{m,n}(a,b,c) = a^2 + b^2 + c^2 - mab - nbc > 1$, then the associated quiver representation is decomposable, by Proposition \ref{Kac}. This means that $M^{\bullet}$ itself, and hence its cohomology sheaf, must also be decomposable, as desired.

\subsection{Decomposability of bundles vs. decomposability of representations}

The last goal of this paper will be to examine under which assumption one does have the converse of the second part of Proposition \ref{decomp-monad}, that is, if the cohomology of a monad is decomposable as a vector bundle, then the quiver representation associated to the monad is also decomposable. The difficulty here, of course, is to argue that if the cohomology of a monad of the form (\ref{mon}) decomposes, then its summands are also cohomologies of monads of the same form. Such statement can be proved under the following additional assumptions, and using the cohomological characterisation of monads provided by Theorem \ref{carmon} above.

Let ${\bf B} = (\boldsymbol{\mathcal{F}}_0, \cdots, \boldsymbol{\mathcal{F}}_n)$ be an $n$-block collection generating the bounded derived category $D^b(X)$ of coherent sheaves on $X$, and let ${}^\vee{\bf B}$ its left dual $n$-block collection, as in the statement of Theorem \ref{carmon}. Let $\mathcal{E}$ be a vector bundle on $X$ given by the cohomology of type (\ref{moncarmon}), and assume that $\mathcal{E}$ is decomposable: $\mathcal{E} \simeq \mathcal{E}_1 \oplus \mathcal{E}_2$. From Theorem \ref{carmon}, since $\mathcal{E}$ has natural cohomology with respect to ${}^\vee{\bf B}$ we have
$$ \dim {\rm Ext}^{n-i-1}(\mathcal{H}^{n-i}_{i_0}, \mathcal{E}) = a, $$
$$ \dim {\rm Ext}^{n-j}(\mathcal{H}^{n-j}_{j_0}, \mathcal{E}) = b, $$
$$ \dim {\rm Ext}^{n-k+1}(\mathcal{H}^{n-k}_{k_0}, \mathcal{E}) = c, $$
and ${\rm ext}^{p}(\mathcal{H}^{m}_{q}, \mathcal{E}) = 0$ otherwise. Hence for $l=1,2$,
$$ \dim {\rm Ext}^{n-i-1}(\mathcal{H}^{n-i}_{i_0}, \mathcal{E}_l) = a_l, $$
$$ \dim {\rm Ext}^{n-j}(\mathcal{H}^{n-j}_{j_0}, \mathcal{E}_l) = b_l, $$
$$ \dim {\rm Ext}^{n-k+1}(\mathcal{H}^{n-k}_{k_0}, \mathcal{E}_l) = c_l,~~ l =1,2, $$
where $a=a_1+a_2$, $b=b_1+b_2$, and $c=c_1+c_2$, with $a_l,b_l, c_l \geq 0$ and
${\rm Ext}^{q}(\mathcal{H}^m_p, \mathcal{E}_l) = 0$ otherwise.

Let us prove that $\mathfrak{M}_{\F^i_{i_0}, \F^j_{j_0}, \F^k_{k_0}}$ is closed under direct summands. From Lemma \ref{beilinsonlemma} and Theorem \ref{carmon}, $\mathcal{E}_l$ is isomorphic to a Beilinson monad $G^{\bullet}_l$, $l=1,2$, where each $G^u_l$ is given by

$$G^u_l = \bigoplus_{p,q}{\rm Ext}^{n-q+u}(\mathcal{H}^{n-q}_p,\mathcal{E}_l) \otimes \F^q_p,\,\, l=1,2.$$ Then we have

$$G^u_l = 0, \, l=1,2; \, u<-1, \,u> 1,$$ and

$$G^{-1}_l = \bigoplus_{p,q} {\rm Ext}^{n-q-1}(\mathcal{H}^{n-q}_{p}, \mathcal{E}_l) \otimes \F^q_p = {\rm Ext}^{n-i-1}(\mathcal{H}^{n-i}_{i_0}, \mathcal{E}_l) \otimes \F^i_{i_0} \simeq (\F^i_{i_0})^{a_l} $$

$$G^{0}_l = \bigoplus_{p,q} {\rm Ext}^{n-q}(\mathcal{H}^{n-q}_{p}, \mathcal{E}_l) \otimes \F^q_p = {\rm Ext}^{n-j}(\mathcal{H}^{n-j}_{j_0}, \mathcal{E}_l) \otimes \F^j_{j_0} \simeq (\F^j_{j_0})^{b_l} $$

$$G^{1}_l = \bigoplus_{p,q} {\rm Ext}^{n-q+1}(\mathcal{H}^{n-q}_{p}, \mathcal{E}_l) \otimes \F^q_p = {\rm Ext}^{n-k+1}(\mathcal{H}^{n-k}_{k_0}, \mathcal{E}_l) \otimes \F^k_{k_0} \simeq (\F^k_{k_0})^{c_l}$$ 

for $l=1,2$. From the definition of Beilinson monad, Definition \ref{Beilinson}, $\mathcal{E}_l$ is isomorphic to the monad

\begin{eqnarray}\label{monadcaract}
\xymatrix{(\F^i_{i_0})^{a_l} \ar[r] & (\F^j_{j_0})^{b_l} \ar[r] & (\F^k_{k_0})^{c_l}}
\end{eqnarray} with $l=1,2$ and $a_l,b_l,c_l \geq 0$.  We have the following cases:

\begin{enumerate}

\item If $a_l, b_l, c_l \neq 0$ for $l=1,2, $ $\mathcal{E}_1$ and $\mathcal{E}_2$ are cohomology of a monad of type (\ref{monadcaract})

$$\mathcal{E}_1 = H^{0}(G^{\bullet}_1),  \,\, \mathcal{E}_2 = H^{0}(G^{\bullet}_2).$$

\item If $a_1,b_1,c_1,b_2,c_2 \neq 0$ and $a_2 = 0$, then $\mathcal{E}_1 = H^{0}(G^{\bullet}_1)$ and $\mathcal{E}_2$ is given by the short exact sequence

$$\xymatrix{0 \ar[r] & \mathcal{E}_2 \ar[r] & (\F^j_{j_0})^{b_2} \ar[r] & (\F^k_{k_0})^{c_2} \ar[r] & 0 .}$$ 

\item If $a_1,b_1,c_1,a_2,b_2 \neq 0$ and $c_2 = 0$ then $\mathcal{E}_1 = H^{0}(G^{\bullet}_1)$ and $\mathcal{E}_2$  is given by the short exact sequence

$$\xymatrix{0 \ar[r] & (\F^i_{i_0})^{a_2} \ar[r] & (\F^j_{j_0})^{b_2} \ar[r] & \mathcal{E}_2 \ar[r] &\ 0 }. $$

\item If $a_1,b_1,c_1,b_2 \neq 0$ and $a_2=c_2 =0$, then $\mathcal{E}_1 = H^{0}(G^{\bullet}_1)$ and $\mathcal{E}_2 \simeq (\F^j_{j_0})^{b_2}$.

\item If $b_1,c_1,a_2,b_2 \neq 0 $ and $a_1 = c_2 = 0$ then

$$\xymatrix{0 \ar[r] & \mathcal{E}_1 \ar[r] & (\F^j_{j_0})^{b_1} \ar[r] & (\F^k_{k_0})^{c_1} \ar[r] & 0 }$$ and

$$\xymatrix{0 \ar[r] & (\F^i_{i_0})^{a_2} \ar[r] & (\F^j_{j_0})^{b_2} \ar[r] & \mathcal{E}_2 \ar[r] &\ 0 }. $$

And the symmetric cases to cases $2,3,4 $ and $5$. 
\end{enumerate} We have just proved that:

\begin{lemma}\label{dirsumm}
The category $\mathfrak{M}_{\F^i_{i_0}, \F^j_{j_0}, \F^k_{k_0}}$ is closed under direct summands.
\end{lemma}

Suppose $m = \dim {\rm Hom}(\F^{i}_{i_0}, \F^j_{j_0})$, and $n = \dim{\rm Hom}(\F^j_{j_0}, \F^k_{k_0})$ and choose $\boldsymbol{\gamma}$ and $\boldsymbol{\sigma}$ bases of ${\rm Hom}(\F^i_{i_0}, \F^j_{j_0})$ and ${\rm Hom}(\F^j_{j_0}, \F^k_{k_0})$, respectively. Let ${\bf G}_{\boldsymbol{\gamma}, \boldsymbol{\sigma}}$ be the functor between $\mathfrak{M}_{\F^i_{i_0}, \F^j_{j_0}, \F^k_{k_0}}$ and $\mathfrak{S}_{m,n}^{\rm gis}$ described after equation (\ref{funtor G}). Note that given a monad of type (\ref{monadcaract}) with $a_l = 0$ the associated representation is 
$$ \xymatrix{
0 \ar@<1.8ex>[r]^{0} \ar@<-1.8ex>[r]^{\vdots}_{0} & \kappa^{b_l } \ar@<-1.8ex>[r]^{\vdots}_{B^l_n} \ar@<1.8ex>[r]^{B^l_1} & \kappa^{c_l}
}$$
which is $(\boldsymbol{\gamma}, \boldsymbol{\sigma})$-globally injective and surjective. If $c_l = 0$, the associated representation is 
$$ \xymatrix{
\kappa^{a_l}
\ar@<1.8ex>[r]^{A^l_1} \ar@<-1.8ex>[r]^{\vdots}_{A^l_m} & \kappa^{b_l } \ar@<-1.8ex>[r]^{\vdots}_{0} \ar@<1.8ex>[r]^{0} & 0
}$$
that is $(\boldsymbol{\gamma},\boldsymbol{\sigma})$-globally injective and surjective. If $a_l= c_l =0$, the associated representation is
$$ \xymatrix{
0 \ar@<1.8ex>[r]^{0} \ar@<-1.8ex>[r]^{\vdots}_{0} & \kappa^{b_l } \ar@<-1.8ex>[r]^{\vdots}_{0} \ar@<1.8ex>[r]^{0} & 0
}$$
which is also $(\boldsymbol{\gamma}, \boldsymbol{\sigma})$-globally injective and surjective. Hence we can prove the following.

\begin{theorem}\label{deccohomology}
Let $\mathcal{E}$ be a vector bundle on $X$ given by the cohomology of a monad in $\mathfrak{M}_{\F^i_{i_0}, \F^j_{j_0}, \F^k_{k_0}}$ and $R$ the associated $(\boldsymbol{\gamma}, \boldsymbol{\sigma})$-globally injective and surjective representation in $\mathfrak{S}_{m,n}^{{\rm gis}}$. Then $\mathcal{E}$ is decomposable if and only if $R$ is decomposable.
\end{theorem}

\begin{proof} We only need to prove the sufficient condition. If $\mathcal{E} \simeq \mathcal{E}_1 \oplus \mathcal{E}_2$ then from Lemma \ref{dirsumm}, $\mathcal{E}_i$, $i=1,2$, are cohomologies of monads in $\mathfrak{M}_{\F^i_{i_0}, \F^j_{j_0}, \F^k_{k_0}}$, therefore $R = {\bf G}_{\boldsymbol{\gamma},\boldsymbol{\sigma}}(\mathcal{E}) \simeq {\bf G}_{\boldsymbol{\gamma},\boldsymbol{\sigma}}(\mathcal{E}_1 \oplus \mathcal{E}_2)\simeq {\bf G}_{\boldsymbol{\gamma}, \boldsymbol{\sigma}}(\mathcal{E}_1) \oplus {\bf G}_{\boldsymbol{\gamma}, \boldsymbol{\sigma}}(\mathcal{E}_2) = R_1 \oplus R_2.$

\end{proof}

%%%%%%%%%%%%%%%%%%%%%%%%%%%%%%%%%%%%%%%%%%%%%%%%%%%%%%%%%%%%%%%%%%%%%%%%%%%%%%%%%%

\subsection{An example: generalized Horrocks--Mumford monads}\label{anexample}

As an application of Theorem \ref{deccohomology}, let $X = \mathbb{P}^{2p}$ with $p\ge 2$, $\kappa=\mathbb{C}$, and consider the $2p$-block collection 
$$ {\bf B} = (\Omega_{\mathbb{P}^{2p}}^{2p}(2p), \Omega_{\mathbb{P}^{2p}}^{2p-1}(2p-1), \cdots, \Omega_{\mathbb{P}^{2p}}^1(1), \mathcal{O}_{\mathbb{P}^{2p}} ) $$
generating the bounded derived category $D^b({\mathbb{P}^{2p}})$. The complex
\begin{eqnarray}\label{monad HM}
\xymatrix{
\mathcal{O}_{\mathbb{P}^{2p}}(-1)^{2p+1} \ar[r]^{\alpha} & \Omega_{\mathbb{P}^{2p}}^p(p)^{2} \ar[r]^{\beta} & \mathcal{O}_{\mathbb{P}^{2p}}^{2p+1} }
\end{eqnarray}
is a monad, for  $\alpha = (\alpha_{ij}) \in \wedge^{p} \mathbb{C}^{2p+1} \otimes {\rm Mat}_{2 \times 2p+1}(\mathbb{C})$ and $\beta = (\beta_{ij}) \in \wedge^{p} \mathcal{C}^{2p+1} \otimes {\rm Mat}_{2p+1 \times 2}(\mathbb{C}) $ given by

\begin{itemize}
\item[] $\beta_{i1} = x_{1+i} \wedge x_{2+i} \wedge  \cdots \wedge x_{p+i};$
\item[] $\beta_{i2} = x_i \wedge x_{p+1+i} \wedge x_{p+2+i} \wedge \cdots \wedge x_{2p-1+i}$

\end{itemize} where $i \equiv k (\rm{mod}~ 2p+1)$ and the matrix $\alpha$ is given by

$$\alpha = (\beta Q)^{t}$$ with 

$$ Q = \left( \begin{array}{cc} 
0 & 1\\
(-1)^{p-1} & 0\\ 
\end{array} \right). $$ 

Note that when $p=2$, the monad (\ref{monad HM}) is precisely the one that yields, as its cohomology, the Horrocks--Mumford rank 2 bundle on $\mathbb{P}^{4}$. For this reason, monads of the form (\ref{monad HM}) are called \emph{generalized Horrocks--Mumford monads}. The goal of this section is to prove, as an application of Theorem \ref{deccohomology}, that the cohomology of a monad of type (\ref{monad HM}) is an indecomposable vector bundle of rank $2\left( \binom{2p}{p} -2p-1 \right)$ on $\mathbb{P}^{2p}$.

To this end, note that one can fix a basis of the vector space $\wedge^p \mathbb{C}^{2p+1}$ so that the quiver representation associated to the morphism
$$ \beta \in \rm{Hom}(\Omega^p_{\mathbb{P}^{2p}}(p)^{2}, \oh_{\mathbb{P}^{2p}}^{2p+1}) $$
is a representation of the Kronecker quiver $K_{\binom{2p+1}{p}}$ of the form
$$ R = (\{\mathbb{C}^2, \mathbb{C}^{2p+1}\}, \{\phi_{l}\}_{l=1}^{\binom{2p+1}{p}}) $$
where $4p+2$ elements $\phi_l$ are elementary matrices of size $(2p+1)\times 2$ for some $l$ and null matrices otherwise. The crucial step is the following result. 

\begin{lemma}
The representation $R$ is simple.
\end{lemma}

In particular, it follows from Theorem \ref{Teo1} that the kernel bundle $\ker\beta$, whose dual is a Steiner bundle, is also simple.

\begin{proof}
Suppose without loss of generality that the ordered basis is the following
$$ \{x_{01\cdots p-1}, x_{12\cdots p},x_{23 \cdots p+1}, \cdots, x_{2p 0 \cdots p-2}, \cdots   \} $$
where $x_{i_1 i_2 \cdots i_{p}} = x_{i_1} \wedge x_{i_2} \wedge \cdots \wedge x_{i_p}$. Then $\beta$ can be written as 

$$\beta =  x_{01\cdots p-1} \cdot E_{2p,1} + x_{12\cdots p}\cdot E_{2p+1,1} +
x_{23 \cdots p+1}\cdot E_{1,1} + \cdots   $$ 
where $E_{i,j} \in \rm{Mat}_{(2p+1)\times 2}(\C)$ is an elementary matrix.  The associated quiver representation is of the form $R = (\{\mathbb{C}^{2}, \mathbb{C}^{2p+1}\}, \{\phi_l\}_{l=1}^{\binom{2p+2}{2}})$ where $\phi_1= E_{2p,1}$, $\phi_2 = E_{2p+1, 1}, \varphi_{3} = E_{1,1}$ and so on. 

 Let $R_1 = (\{V_1,V_2\}, \{\psi\}_{l=1}^{\binom{2p+1}{p}})$ be a subrepresentation of $R$ and without loss of generality suppose $V_1 \neq 0$. Then there is $v = (a,b) \in V_1 \subset \C^2, $ with $a\neq 0$ . The following diagram commutes

\begin{eqnarray}%\label{morphism2}
\xymatrix{\C^{2} \ar@<1.8ex>[r]^{\phi_1} \ar@<-1.8ex>[r]^>>>>{\vdots}_{\phi_{\binom{2p+1}{p}}} & \C^{2p+1} \\
& &  \\
V_1 \ar[uu]^{i_1}\ar@<1.8ex>[r]^{\psi_1} \ar@<-1.8ex>[r]^<<<<<<{\vdots}_{\psi_{\binom{2p+1}{p}}} & V_{2}\ar[uu]_{i_2}}
\end{eqnarray} and note that the vectors $\{\phi_j(v)\}_{j=1}^{2p+1}$ are linearly independent, hence $V_2 \simeq \C^{2p+1}$.
If $R_2 = (\{W_1,W_2\}, \{\gamma_l\}_{l=1}^{\binom{2p+1}{p}})$ is a subrepresentation of $R$ such that $R \simeq R_1 \oplus R_2$, then $W_2 \equiv 0$ and if $W_1 \neq 0$ there is $k$ such that $\phi_k \neq 0$ and ${\phi_k}\mid_{W_1} \neq 0$. Therefore $R$ is simple, hence indecomposable.
\end{proof}

Since $\alpha = (\beta Q)^t$, the representation of the Kronecker quiver $K_{\binom{2p+1}{p}}$ associated to $\alpha$ is of the form $R^{'}=(\{\C^{2p+1}, \C^2\}, \{\phi'_l \}_{l=1}^{\binom{2p+1}{p}}),$ where $\phi'_l$ are elementary matrices or null matrices (the transpose of $\phi_l$ up to sign). Hence $R'$ is also simple. By Theorem \ref{deccohomology}, the cohomology of the monad  (\ref{monad HM}) is an indecomposable vector bundle on $\mathbb{P}^{2p}$.

%%%%%%%%%%%%%%%%%%%%%%%%%%%%%%%%%%%%%%%%% bibliography %%%%%%%%%%%%%%%%%%%%%%%%%%%%%%%%%%%%%%%%%%
%%%%%%%%%%%%%%%%%%%%%%%%%%%%%%%%%%%%%%%%%%%%%%%%%%%%%%%%%%%%%%%%%%%%%%%%%%%%%%%%%%%%%%%%%%%%%%%%%

\end{document}